\documentclass[a4paper,12pt]{article}

\usepackage{authblk}
\usepackage{parskip}
\usepackage{mathtools}
\usepackage{amssymb}
\usepackage{amsmath}
\usepackage{latexsym}
\usepackage{mathrsfs}  
\usepackage{bbm}       
\usepackage{amsthm}    
\usepackage{relsize}   
\usepackage{thmtools}  
\usepackage{mathrsfs}  
\usepackage{lipsum}    
\usepackage{hyperref}
\usepackage{orcidlink}

\numberwithin{equation}{section}

\usepackage{titlesec}   

\titleformat{\section}[block]{\bfseries\filcenter}
{{\upshape\thesection\enspace}}{.5em}{}

\titleformat{\subsection}[block]{\filcenter}
{{\upshape\thesubsection\enspace}}{.5em}{} 

\usepackage{enumitem}  
\setlist{nosep}  
\setitemize[0]{leftmargin=*}
\setenumerate[0]{leftmargin=*}
\setenumerate[1]{label={(\arabic*)}} 

\newcommand{\N}{\mathbb{N}}     
\newcommand{\R}{\mathbb{R}}     
\newcommand{\Prob}{\mathbb{P}}  
\newcommand{\Exp}{\mathbb{E}}   
\newcommand{\goth}[1]{\mathfrak{#1}} 
\newcommand{\inner}[2]{\left\langle #1 \, , \, #2 \right\rangle} 
\newcommand{\norm}[1]{\left|\left|#1\right|\right|}              
\newcommand{\triplet}[3]{\left( #1, #2, #3 \right) }             
\newcommand{\ProbSpace}{\triplet{\Omega}{\mathscr{F}}{\Prob}}    
\newcommand{\abs}[1]{\left| #1 \right|}                          

\newcommand{\defeq}{\mathrel{\mathop:}=}                         
\newcommand\restr[2]{{
  \left.\kern-\nulldelimiterspace 
  #1 
  \vphantom{\big|} 
  \right|_{#2} 
  }}
  
\makeatletter
\newsavebox{\@brx}
\newcommand{\llangle}[1][]{\savebox{\@brx}{\(\m@th{#1\langle}\)}%
  \mathopen{\copy\@brx\kern-0.5\wd\@brx\usebox{\@brx}}}
\newcommand{\rrangle}[1][]{\savebox{\@brx}{\(\m@th{#1\rangle}\)}%
  \mathclose{\copy\@brx\kern-0.5\wd\@brx\usebox{\@brx}}}
\makeatother

\theoremstyle{plain} 
\newtheorem{theorem}{Theorem}[section]    
\newtheorem{proposition}[theorem]{Proposition} 

\newtheorem{lemma}[theorem]{Lemma}
\newtheorem{assumption}[theorem]{Assumption}

\theoremstyle{definition} 

\newtheorem{example}[theorem]{Example}
\newtheorem{remark}[theorem]{Remark}

 \title{\Large Tightness and weak convergence in the topology of local uniform convergence for stochastic processes in the dual of a nuclear space}
\author{C. A. Fonseca-Mora \orcidlink{0000-0002-9280-8212}}
\affil{Centro de Investigaci\'{o}n en Matem\'{a}tica Pura y Aplicada, \\ Escuela de Matem\'{a}tica, Universidad de Costa Rica. \\
\noindent E-mail:  christianandres.fonseca@ucr.ac.cr }
\date{}

\begin{document}

 \maketitle

\abstract{Let $\Phi'$ denote the strong dual of a nuclear space $\Phi$ and let $C_{\infty}(\Phi')$ be the collection of all continuous mappings $x:[0,\infty) \rightarrow \Phi'$ equipped with the topology of local uniform convergence. In this paper we prove sufficient conditions for tightness of probability measures on  $C_{\infty}(\Phi')$ and for weak convergence in   $C_{\infty}(\Phi')$ for a sequence of $\Phi'$-valued processes. We illustrate our results with two applications. First, we show the central limit theorem for local martingales taking values in the dual of an ultrabornological nuclear space. Second, we prove sufficient conditions for the weak convergence in $C_{\infty}(\Phi')$ for a sequence of solutions to stochastic partial differential equations driven by semimartingale noise.}

\smallskip



\section{Introduction}

Tightness and weak convergence are two central topics within the theory of stochastic processes and are of great importance both from a theoretical perspective as well for applications. The main objective of this paper is to prove sufficient conditions for tightness of probability measures on the space of continuous mappings taking values in the dual of a nuclear space, as well to prove sufficient conditions for weak convergence in the local uniform topology for a sequence of continuous stochastic processes taking values in the dual of a nuclear space. We are particularly interested in the formulation of the aforementioned results under the most general conditions. 

In the context of stochastic processes taking values in the strong dual $\Phi'$ of a Fr\'echet nuclear space $\Phi$ (in particular for $\mathscr{S}'(\R^{d})$-valued processes), the first results on tightness and weak convergence were obtained by Mitoma in \cite{Mitoma:1983UT}. In particular, Mitoma shown that in order to prove tightness of a family of probability measures defined on the space $C_{\infty}(\Phi')$ of continuous $\Phi'$-valued mappings, a sufficient condition is tightness of the corresponding one dimensional projections on the space $C_{\infty}(\R)$ of continuous real-valued mappings. In the case of weak convergence of a sequence of continuous $\Phi'$-valued processes, apart from tightness of one dimensional projections on  $C_{\infty}(\R)$ it is required the convergence in distribution of finite dimensional projections. One of the main contributions of Mitoma's work is that it pointed out the strong finite dimensional character of nuclear spaces and its dual spaces, and how this property has important implications in the study of tightness and weak convergence of stochastic processes.  

In \cite{Fouque:1984}, Fouque carried out an extension of the results of Mitoma to the case when $\Phi$ is either the strong dual to a nuclear Fr\'{e}chet space or a countable inductive limit of nuclear Fr\'{e}chet spaces (in particular to $\mathscr{D}'(\R^{d})$-valued processes). 

Although the results of Mitoma and Fouque have found many applications, see \cite{BojdeckiGorostizaTalarczyk:2008, BudhirajaWu:2017, Chen:2019, GubinelliTurra:2020, Lucon:2020} to cite but a few, we are not aware of any attempt to extend these results to a larger class of nuclear spaces.
In this paper we will try to fill this gap on the literature to extend the aforementioned results to the to the setting of cylindrical stochastic processes taking values in the dual of a  general nuclear space $\Phi$. As we will explain in the next section, our results are not only important for the sake of generality, but they also have extensive theoretical and practical application. 

\subsection{Statement of main results}\label{sectMainTheorems}

Let $\Phi$ be a nuclear space with strong dual $\Phi'$. Denote by  $C_{\infty}(\Phi')$ the collection of all continuous mappings $x:[0,\infty) \rightarrow \Phi'$ equipped with the topology of local uniform convergence. 

We now proceed to state the main theorems of this paper. The basic background and terminology   on cylindrical stochastic processes and the topology of local uniform convergence, as it used in the statements of the theorems below, are given in Section \ref{sectionPrelim}. We start with our main result on uniform tightness of probability measures on $C_{\infty}(\Phi')$.

\begin{theorem}\label{theoUniforTightSpaceContiTimeInfty}
Assume that $(\mu_{\alpha}: \alpha \in A)$ is a family of probability measures on $C_{\infty}(\Phi')$ such that:
\begin{enumerate}
\item \label{condEquiFourTrans} For all $T>0$, the family of Fourier transforms $(\hat{\mu}_{\alpha,t}: t \in [0,T], \alpha \in A)$ is equicontinuous at zero. 
\item \label{condFinDimenTight} For each $\phi \in \Phi$, the family $(\mu_{\alpha} \circ \Pi^{-1}_{\phi}: \alpha \in A)$ of probability measures on $C_{\infty}(\R)$ is uniformly tight. 
\end{enumerate}
Then there exists a weaker countably Hilbertian topology $\theta$ on $\Phi$ such that  $(\mu_{\alpha}: \alpha \in A)$ is uniformly tight on $C_{\infty}((\widehat{\Phi}_{\theta})')$, and in particular is uniformly tight on $C_{\infty}(\Phi')$. 
\end{theorem}

\begin{remark}
If $\Phi$ is a barrelled nuclear space, the conditions \ref{condEquiFourTrans} and \ref{condFinDimenTight} in Theorem \ref{theoUniforTightSpaceContiTimeInfty} are necessary for uniform tightness of a family $(\mu_{\alpha}: \alpha \in A)$ of probability measures on $C_{\infty}(\Phi')$. We will explain this later in Remark \ref{remaNecessaryCondTighness}. 
\end{remark}

\begin{remark}
If $\Phi$ is an ultrabornological nuclear space, then condition \ref{condFinDimenTight} implies condition \ref{condEquiFourTrans}  in Theorem \ref{theoUniforTightSpaceContiTimeInfty}. 
This fact will be explained later in Remark \ref{remaUltraboTheoTight}. 
We therefore have that 
Theorem \ref{theoUniforTightSpaceContiTimeInfty} generalizes 
Theorem 3.1 of Mitoma in \cite{Mitoma:1983UT}, there formulated under the assumption that $\Phi$ is a Fr\'{e}chet nuclear space. Moreover, Theorem \ref{theoUniforTightSpaceContiTimeInfty} also generalizes the results in Proposition III.1 and Th\'{e}or\`{e}me IV.1 of Fouque in \cite{Fouque:1984}, there formulated respectively for the strong dual to a nuclear Fr\'{e}chet space and for the strict inductive limit of a sequence of nuclear Fr\'{e}chet spaces.  
\end{remark}

\begin{remark}
Condition \ref{condFinDimenTight} in Theorem \ref{theoUniforTightSpaceContiTimeInfty} was already present in the work of Mitoma and Fouque. On the other hand, the condition  \ref{condFinDimenTight} is new and plays a fundamental role in the context of a general nuclear space $\Phi$. As we will see in the proof of Theorem \ref{theoUniforTightSpaceContiTimeInfty} in Section \ref{sectProofMainResults}, the condition  \ref{condFinDimenTight} allow us to make use of the regularization theorem (see Theorem 3.2 in \cite{FonsecaMora:Existence}) which in this context is fundamental to show that each $\mu_{\alpha}$ defines a Radon probability measure on  $C_{\infty}((\widehat{\Phi}_{\theta})')$. Combined with condition \ref{condFinDimenTight} the above shows that  $(\mu_{\alpha}: \alpha \in A)$ is uniformly tight on $C_{\infty}((\widehat{\Phi}_{\theta})')$. The existence of a weaker countably Hilbertian topology $\theta$ is of great importance as it settles our problem of checking tightness in the context of continuous functions with values in the dual of a metrizable space. 
\end{remark}

Our main result on weak convergence is the following. It gives sufficient conditions for weak convergence in $C_{\infty}(\Phi')$ for a sequence of cylindrical processes  in $\Phi'$.

\begin{theorem}\label{theoWeakConvergCylinProcesInCTInfy}
Let $\Phi$ be a nuclear space. Let $(X^{n}: n \in \N)$, with $X^{n} =(X^{n}_{t}: t \geq 0)$, be a sequence of cylindrical process in $\Phi'$ satisfying:
\begin{enumerate}
\item For every $n \in \N$ and $\phi \in \Phi$ the real-valued process $X^{n}(\phi)=( X^{n}_{t}(\phi): t \geq 0)$ is continuous.
\item \label{condEquiTheoWeakConv} For every $T > 0$, the family $( X^{n}_{t}: t \in [0,T], n \in \N )$ of linear maps from $\Phi$ into $L^{0} \ProbSpace$ is equicontinuous at zero.   
\item \label{condTighTheoWeakConv} For each $\phi \in \Phi$, the sequence of distributions of $X^{n}(\phi)$ is uniformly tight on $C_{\infty}(\R)$.
\item  $\forall$ $m \in \N$, $\phi_{1}, \dots, \phi_{m} \in \Phi$, $t_{1}, \dots, t_{m} \geq 0$, the distribution of the $m$-dimensional random vector $(X_{t_{1}}^{n}(\phi_{1}), \cdots, X_{t_{m}}^{n}(\phi_{m}))$ converges in distribution to some probability measure on $\R^{m}$. 
\end{enumerate}
Then there exist a weaker countably Hilbertian topology $\theta$ on $\Phi$ and some  $C_{\infty}((\widehat{\Phi}_{\theta})')$-valued random variables  $Y$ and  $Y^{n}$, $n \in \N$,  such that 
\begin{enumerate}[label=(\roman*)]
\item \label{concluRegularVersTheoWeakConv} For every $\phi \in \Phi$ and $n \in \N$, the real-valued processes $\inner{Y^{n}}{\phi}$ and $X^{n}(\phi)$ are indistinguishable.
\item \label{concluTightTheoWeakConv} The sequence $(Y^{n}:n \in \N)$ is  uniformly tight on $C_{\infty}((\widehat{\Phi}_{\theta})')$.
\item \label{concluConvergenTheoWeakConv} $Y^{n} \Rightarrow Y$ in $C_{\infty}((\widehat{\Phi}_{\theta})')$
\end{enumerate}
Moreover, \ref{concluTightTheoWeakConv} and \ref{concluConvergenTheoWeakConv} are also satisfied for $Y$ and $(Y^{n}: n \in \N)$ as $C_{\infty}(\Phi')$-valued random variables. 
\end{theorem}

\begin{remark}
The result of Theorem \ref{theoWeakConvergCylinProcesInCTInfy} can be interpreted as follows: if we have a sequence of cylindrical processes $(X^{n}:n \in \N)$ satisfying the hypothesis of the theorem, then one can find a sequence of $\Phi'$-valued processes $(Y^{n}:n \in \N)$ with continuous paths such that each $Y^{n}$ is a regular version of $X^{n}$ and such that $(Y^{n}:n \in \N)$ converges weakly in $C_{\infty}(\Phi')$. 
Thus Theorem \ref{theoWeakConvergCylinProcesInCTInfy} is both a result on regularization as well a result on weak convergence in $C_{\infty}(\Phi')$.
\end{remark}

\begin{remark} 
At a first glance, the formulation of  Theorem \ref{theoWeakConvergCylinProcesInCTInfy} in the context of cylindrical processes seems a bit abstract for practical applications. However, the cylindrical context is very useful when it comes to formulate new theoretical results. For example, in \cite{FonsecaMora:UCPConvergence} we apply Theorem \ref{theoWeakConvergCylinProcesInCTInfy} to introduce sufficient conditions for the convergence uniform on compacts in probability of a sequence of $\Phi'$-valued stochastic process. 
\end{remark}

We can reformulate the result of Theorem \ref{theoWeakConvergCylinProcesInCTInfy} in a context that makes it useful for practical applications. This is done below under  the assumption that $\Phi$ is an ultrabornological nuclear space (of which the classical spaces of test functions and distributions are a particular example). 

\begin{theorem}\label{theoWeakConvUltrabornological}
Let $\Phi$ be an ultrabornological nuclear space. For each $n \in \N$, let $X^{n}=(X^{n}_{t}: t \geq 0)$ be a $\Phi'$-valued regular process with continuous paths satisfying:
\begin{enumerate}
\item  For each $\phi \in \Phi$, the sequence of distributions of $\inner{X^{n}}{\phi}$ is uniformly tight on $C_{\infty}(\R)$.
\item  $\forall$ $m \in \N$, $\phi_{1}, \dots, \phi_{m} \in \Phi$, $t_{1}, \dots, t_{m} \geq 0$, the distribution of the $m$-dimensional random vector $(\inner{X_{t_{1}}^{n}}{\phi_{1}}, \cdots, \inner{X_{t_{m}}^{n}}{\phi_{m}})$ converges in distribution to some probability measure on $\R^{m}$. 
\end{enumerate} 
Then, there exists a $\Phi'$-valued regular process $X=(X_{t}: \geq 0)$ with continuous paths such that $X^{n} \Rightarrow X$ in $C_{\infty}(\Phi')$. 
\end{theorem}

\begin{remark}
Theorem \ref{theoWeakConvUltrabornological} generalizes Theorem 5.3 of Mitoma \cite{Mitoma:1983UT}, there formulated under the assumption that $\Phi$ is a Fr\'{e}chet nuclear space. Moreover, Theorem \ref{theoWeakConvUltrabornological} also generalizes Corollarie IV.1 of Fouque in \cite{Fouque:1984}, where $\Phi$ is assumed to be the strict inductive limit of a sequence of nuclear Fr\'{e}chet spaces.
\end{remark}

\begin{remark}
An analogue to Theorems \ref{theoUniforTightSpaceContiTimeInfty}, \ref{theoWeakConvergCylinProcesInCTInfy} and \ref{theoWeakConvUltrabornological} 
but for tightness and weak convergence in the Skorokhod space is proved by the author in (\cite{FonsecaMora:Skorokhod}).
\end{remark}
  
In this paper we give two applications for Theorem  \ref{theoWeakConvUltrabornological} . First, we introduce sufficient conditions for a sequence of continuous local martingales to converge in $C_{\infty}(\Phi')$ (see Theorem \ref{theoWeakConvContLocalMartin}). As an consequence of our result, we show a version of the central limit theorem for local martingales taking values in the dual of an ultrabornological nuclear space (see Theorem  \ref{theoCentralLimitTheorem}). This last result generalizes previous work carried our by Fierro in \cite{Fierro:1988} for local martingales taking values in the dual of a nuclear Fr\'echet space. 

For our second application, for each $n=0,1,2, \dots$ we consider a linear stochastic partial differential equation of the form 
$$
dY^{n}_{t}=(A^{n})'Y^{n}_{t} dt +dZ^{n}_{t}, 
$$
with initial condition $Y^{n}_{0}=\eta^n$; where $A^{n}$ is the infinitesimal generator of a $C_{0}$-semigroup of continuous linear operators on $\Phi'$ and $Z^{n}$ is a continuous semimartingale in $\Phi'$. Denote by $Y^{n}$  the (unique) weak solution to such initial value problem. We apply our result on weak convergence to show that under some basic conditions on $\eta^{n}$, $A^{n}$ and $Z^{n}$ (see Theorem \ref{theoWeakConvSoluSEE}), we have $Y^{n} \Rightarrow Y^{0}$ in $C_{\infty}(\Phi')$ as $n \rightarrow \infty$.

One of the main motivations for the results of this paper is  the recent progress in the theory of stochastic analysis with respect to semimartingales (see \cite{FonsecaMora:StochInteg}). In particular, we are interested in to apply our results to show existence of solutions to non-linear stochastic partial differential equations driven by multiplicative semimartingale noise. These results will appear elsewhere.

\subsection{Organization of the paper} The paper is organized as follows.  In Section \ref{sectionPrelim} we list some important notions concerning nuclear spaces and their duals, cylindrical and stochastic processes. We also introduce the basic properties of the local uniform topology on the space of continuous mappings. The proofs of Theorems 
\ref{theoUniforTightSpaceContiTimeInfty}, \ref{theoWeakConvergCylinProcesInCTInfy} and \ref{theoWeakConvUltrabornological} can be found in Section \ref{sectProofMainResults}. The central limit theorem for local martingales is the object of study of Section \ref{sectCentralLimitTheorem}. The study of weak convergence of a sequence of solutions to stochastic partial differential equations with semimartingale noise is carried our in Section  \ref{secWeakSemimart}. Finally, in Section \ref{sectExample} we give a detailed example of our main results.

\section{Preliminaries} \label{sectionPrelim}

\subsection{Nuclear spaces and their strong duals}

Let $\Phi$ be a locally convex space (we will only consider vector spaces over $\R$). $\Phi$ is called \emph{bornological} (respectively,  \emph{ultrabornological}) if it is the inductive limit of a family of normed (respectively, Banach) spaces. A \emph{barreled space} is a locally convex space such that every convex, balanced, absorbing and closed subset is a neighborhood of zero. For further details see \cite{Jarchow, NariciBeckenstein}.   

If $p$ is a continuous semi-norm on $\Phi$ and $r>0$, the closed ball of radius $r$ of $p$, given by $B_{p}(r) = \left\{ \phi \in \Phi: p(\phi) \leq r \right\}$, is a closed, convex, balanced neighborhood of zero in $\Phi$. 
A continuous seminorm (respectively norm) $p$ on $\Phi$ is called \emph{Hilbertian} if $p(\phi)^{2}=Q(\phi,\phi)$, for all $\phi \in \Phi$, where $Q$ is a symmetric non-negative bilinear form (respectively inner product) on $\Phi \times \Phi$. For any given continuous seminorm $p$ on $\Phi$ let $\Phi_{p}$ be the Banach space that corresponds to the completion of the normed space $(\Phi / \mbox{ker}(p), \tilde{p})$, where $\tilde{p}(\phi+\mbox{ker}(p))=p(\phi)$ for each $\phi \in \Phi$. We denote by  $\Phi'_{p}$ the dual to the Banach space $\Phi_{p}$ and by $p'$ the corresponding dual norm. Observe that if $p$ is Hilbertian then $\Phi_{p}$ and $\Phi'_{p}$ are Hilbert spaces. If $q$ is another continuous seminorm on $\Phi$ for which $p \leq q$, we have that $\mbox{ker}(q) \subseteq \mbox{ker}(p)$ and the canonical inclusion map from $\Phi / \mbox{ker}(q)$ into $\Phi / \mbox{ker}(p)$ has a unique continuous and linear extension that we denote by $i_{p,q}:\Phi_{q} \rightarrow \Phi_{p}$. Furthermore, we have the following relation: $i_{p}=i_{p,q} \circ i_{q}$.

We denote by $\Phi'$ the topological dual of $\Phi$ and by $\inner{f}{\phi}$ the canonical pairing of elements $f \in \Phi'$, $\phi \in \Phi$. Unless otherwise specified, $\Phi'$ will always be consider equipped with its \emph{strong topology}, i.e. the topology on $\Phi'$ generated by the family of semi-norms $( \eta_{B} )$, where for each $B \subseteq \Phi$ bounded, $\eta_{B}(f)=\sup \{ \abs{\inner{f}{\phi}}: \phi \in B \}$ for all $f \in \Phi'$.  

Let $p$ and $q$ be continuous Hilbertian semi-norms on $\Phi$ such that $p \leq q$. The space of continuous linear operators (respectively Hilbert-Schmidt operators) from $\Phi_{q}$ into $\Phi_{p}$ is denoted by $\mathcal{L}(\Phi_{q},\Phi_{p})$ (respectively $\mathcal{L}_{2}(\Phi_{q},\Phi_{p})$). We employ an analogous notation for operators between the dual spaces $\Phi'_{p}$ and $\Phi'_{q}$. 

Let us recall that a (Hausdorff) locally convex space $(\Phi,\mathcal{T})$ is called \emph{nuclear} if its topology $\mathcal{T}$ is generated by a family $\Pi$ of Hilbertian semi-norms such that for each $p \in \Pi$ there exists $q \in \Pi$, satisfying $p \leq q$ and the canonical inclusion $i_{p,q}: \Phi_{q} \rightarrow \Phi_{p}$ is Hilbert-Schmidt. Other equivalent definitions of nuclear spaces can be found in \cite{Pietsch, Treves}. 

Let $\Phi$ be a nuclear space. If $p$ is a continuous Hilbertian semi-norm  on $\Phi$, then the Hilbert space $\Phi_{p}$ is separable (see \cite{Pietsch}, Proposition 4.4.9 and Theorem 4.4.10, p.82). Now, let $( p_{n} : n \in \N)$ be an increasing sequence of continuous Hilbertian semi-norms on $(\Phi,\mathcal{T})$. We denote by $\theta$ the locally convex topology on $\Phi$ generated by the family $( p_{n} : n \in \N)$. The topology $\theta$ is weaker than $\mathcal{T}$. We  will call $\theta$ a (weaker) \emph{countably Hilbertian topology} on $\Phi$ and we denote by $\Phi_{\theta}$ the space $(\Phi,\theta)$ and by $\widehat{\Phi}_{\theta}$ its completion. The space $\widehat{\Phi}_{\theta}$ is a (not necessarily Hausdorff) separable, complete, pseudo-metrizable (hence Baire and ultrabornological; see Example 13.2.8(b) and Theorem 13.2.12 in \cite{NariciBeckenstein}) locally convex space and its dual space satisfies $(\widehat{\Phi}_{\theta})'=(\Phi_{\theta})'=\bigcup_{n \in \N} \Phi'_{p_{n}}$ (see \cite{FonsecaMora:Existence}, Proposition 2.4). 

The class of nuclear spaces includes many spaces of functions widely used in analysis. Indeed, it is known (see e.g. \cite{Pietsch, Schaefer, Treves}) that the spaces of test functions $\mathscr{E}_{K} \defeq \mathcal{C}^{\infty}(K)$ ($K$: compact subset of $\R^{d}$), $\mathscr{E}\defeq \mathcal{C}^{\infty}(\R^{d})$, the rapidly decreasing functions $\mathscr{S}(\R^{d})$, and the space of harmonic functions $\mathcal{H}(U)$ ($U$: open subset of $\R^{d}$),  are all  examples of Fr\'{e}chet nuclear spaces. Their (strong) dual spaces $\mathscr{E}'_{K}$, $\mathscr{E}'$, $\mathscr{S}'(\R^{d})$, $\mathcal{H}'(U)$, are also nuclear spaces.
On the other hand, the space of test functions $\mathscr{D}(U) \defeq \mathcal{C}_{c}^{\infty}(U)$ ($U$: open subset of $\R^{d}$), the space of polynomials $\mathcal{P}_{n}$ in $n$-variables, the space of real-valued sequences $\R^{\N}$ (with direct sum topology) are strict inductive limits of Fr\'{e}chet nuclear spaces (hence they are also nuclear). The space of distributions  $\mathscr{D}'(U)$  ($U$: open subset of $\R^{d}$) is also nuclear.   
All the above are examples of (complete) ultrabornological nuclear spaces.  

\subsection{Cylindrical and stochastic processes}
Let $E$ be a topological space and denote by $\mathcal{B}(E)$ its Borel $\sigma$-algebra. Recall that a Borel measure $\mu$ on $E$ is called a \emph{Radon measure} if for every $\Gamma \in \mathcal{B}(E)$ and $\epsilon >0$, there exist a compact set $K \subseteq \Gamma$ such that $\mu(\Gamma \backslash K) < \epsilon$. We denote by $\goth{M}_{R}^{1}(E)$ the space of all Radon probability measures on $E$.

Throughout this work we assume that $\ProbSpace$ is a complete probability space and consider a filtration $( \mathcal{F}_{t} : t \geq 0)$ on $\ProbSpace$ that satisfies the \emph{usual conditions}, i.e. it is right continuous and $\mathcal{F}_{0}$ contains all subsets of sets of $\mathcal{F}$ of $\Prob$-measure zero. The space $L^{0} \ProbSpace$ of equivalence classes of real-valued random variables defined on $\ProbSpace$ will always be considered equipped with the topology of convergence in probability and in this case it is a complete, metrizable, topological vector space.

Let $\Phi$ be a locally convex space. A \emph{cylindrical random variable}\index{cylindrical random variable} in $\Phi'$ is a linear map $X: \Phi \rightarrow L^{0} \ProbSpace$ (see \cite{FonsecaMora:Existence}). Let $X$ be a $\Phi'$-valued random variable, i.e. $X:\Omega \rightarrow \Phi'$ is a $\mathscr{F}/\mathcal{B}(\Phi')$-measurable map. For each $\phi \in \Phi$ we denote by $\inner{X}{\phi}$ the real-valued random variable defined by $\inner{X}{\phi}(\omega) \defeq \inner{X(\omega)}{\phi}$, for all $\omega \in \Omega$. The linear mapping $\phi \mapsto \inner{X}{\phi}$ is called the \emph{cylindrical random variable induced/defined by} $X$.


Let $J=\R_{+} \defeq [0,\infty)$ or $J=[0,T]$ for  $T>0$. We say that $X=( X_{t}: t \in J)$ is a \emph{cylindrical process} in $\Phi'$ if $X_{t}$ is a cylindrical random variable for each $t \in J$. Clearly, any $\Phi'$-valued stochastic processes $X=( X_{t}: t \in J)$ induces/defines a cylindrical process under the prescription: $\inner{X}{\phi}=( \inner{X_{t}}{\phi}: t \in J)$, for each $\phi \in \Phi$. 

If $X$ is a cylindrical random variable in $\Phi'$, a $\Phi'$-valued random variable $Y$ is called a \emph{version} of $X$ if for every $\phi \in \Phi$, $X(\phi)=\inner{Y}{\phi}$ $\Prob$-a.e. A $\Phi'$-valued process $Y=(Y_{t}:t \in J)$ is said to be a $\Phi'$-valued \emph{version} of the cylindrical process $X=(X_{t}: t \in J)$ on $\Phi'$ if for each $t \in J$, $Y_{t}$ is a $\Phi'$-valued version of $X_{t}$.

A $\Phi'$-valued process $X=( X_{t}: t \in J)$ is  \emph{continuous} (respectively \emph{c\`{a}dl\`{a}g}) if for $\Prob$-a.e. $\omega \in \Omega$, the \emph{sample paths} $t \mapsto X_{t}(\omega) \in \Phi'$ of $X$ are continuous (respectively c\`{a}dl\`{a}g). 


A $\Phi'$-valued random variable $X$ is called \emph{regular} if there exists a weaker countably Hilbertian topology $\theta$ on $\Phi$ such that $\Prob( \omega: X(\omega) \in (\widehat{\Phi}_{\theta})')=1$. If $\Phi$ is a barrelled (e.g. ultrabornological) nuclear space, the property of being regular is  equivalent to the property that the law of $X$ be a Radon measure on $\Phi'$ (see Theorem 2.10 in \cite{FonsecaMora:Existence}).  A $\Phi'$-valued process $X=(X_{t}:t \geq 0)$ is said to be \emph{regular} if for each $t \geq 0$, $X_{t}$ is a regular random variable.

\subsection{The space of continuous mappings with the topology of local uniform convergence}\label{sectSpaceContMappins}

Let $\Phi$ denote a Hausdorff locally convex space and let $T>0$. We denote by $C_{T}(\Phi')$ the linear space of all the $\Phi'$-valued continuous mappings $x:[0,T] \rightarrow \Phi'$, equipped with the \emph{topology of uniform convergence on $[0,T]$}, that is, the topology generated by the seminorms
$$ x \mapsto \sup_{t \in [0,T]} p(x(t)), \quad \forall p \in \Pi,$$
where $\Pi$ is a family of continuous seminorms generating the topology on $\Phi'$. 


We will require the following notions of cylindrical random variable and cylindrical measure on $C_{T}(\Phi')$. 
Given $m \in \N$, $\phi_{1}, \dots, \phi_{m} \in \Phi$, $t_{1}, \dots, t_{m} \in [0,T]$,  we define the \emph{space-time projection map} 
$\Pi^{\phi_{1}, \dots, \phi_{m}}_{t_{1}, \dots, t_{m}}: C_{T}(\Phi') \rightarrow \R^{m}$ by 
\begin{equation*}
\Pi^{\phi_{1}, \dots, \phi_{m}}_{t_{1}, \dots, t_{m}}(x)  =
( \inner{x(t_{1})}{\phi_{1}}, \dots, \inner{x(t_{m})}{\phi_{m}}), \quad \forall \, x \in C_{T}(\Phi').
\end{equation*} 
If $M=\{\phi_{1}, \dots, \phi_{m} \} \subseteq  \Phi $, $I=\{ 
t_{1}, \dots, t_{m} \} \subseteq  [0,T]$ and $B \in \mathcal{B}(\R^{m})$, the set $\mathcal{Z}(M,I,B) = \left( \Pi^{\phi_{1}, \dots, \phi_{m}}_{t_{1}, \dots, t_{m}} \right)^{-1}(B) $
is called a \emph{cylinder set} in $C_{T}(\Phi')$ based on $(M,I)$. 
 The collection 
$ \mathscr{C}(C_{T}(\Phi');M,I)=\{ \mathcal{Z}(M,I,B): B \in \mathcal{B}(\R^{m})\}$ is a $\sigma$-algebra, called the \emph{cylindrical $\sigma$-algebra in $C_{T}(\Phi')$ based on} $(M,I)$.
The \emph{cylindrical $\sigma$-algebra}  
$ \mathscr{C}(C_{T}(\Phi'))$ is the $\sigma$-algebra generated by the cylinder sets. We have $\mathscr{C}(C_{T}(\Phi')) \subseteq  \mathcal{B}(C_{T}(\Phi'))$ and in some cases we have equality (see Proposition \ref{propProperContiSpaceWeakerCHT} below).

A \emph{cylindrical (probability) measure} on $C_{T}(\Phi')$ is a map $\mu: \mathcal{Z}(C_{T}(\Phi')) \rightarrow [0,+\infty]$ such that for  finite $M \subseteq \Phi$ and $I \subseteq [0,T]$, the restriction of $\mu$ to  $\mathscr{C}(C_{T}(\Phi');M,I)$ is a (probability) measure. Any probability measure on $C_{T}(\Phi')$ is a cylindrical measure. 

 A \emph{cylindrical random variable} in $C_{T}(\Phi')$ (defined on a probability space $\ProbSpace$) is a linear map $X: \Phi \rightarrow L^{0}(\Omega, \mathcal{F}, \Prob;C_{T}(\R))$, where $L^{0}(\Omega, \mathcal{F}, \Prob;C_{T}(\R))$ is the set of $C_{T}(\R)$-valued random variables defined on $\ProbSpace$.
 
The \emph{cylindrical distribution} of $X$ is the cylindrical probability measure on $C_{T}(\Phi')$ satisfying:
$$\mu_{X} \left( \left( \Pi^{\phi_{1}, \dots, \phi_{m}}_{t_{1}, \dots, t_{m}} \right)^{-1}(B) \right) = \Prob \circ X^{-1} \circ \left( \Pi^{\phi_{1},\dots, \phi_{m}}_{t_{1}, \dots, t_{m}}\right)^{-1}(B), $$
for every $m \in \N$, $\phi_{1}, \dots, \phi_{m} \in  \Phi$, $t_{1}, \dots, t_{m} \in [0,T]$ and $B \in \mathcal{B}(\R^{m})$. By following the same arguments used in the proof of Theorem 4.5 in \cite{FonsecaMora:Skorokhod} (there for the Skorokhod space) it can be shown that to every cylindrical probability measure on $C_{T}(\Phi')$ there corresponds a canonical  cylindrical random variable in $C_{T}(\Phi')$.

Given $\phi \in \Phi$, we define the \emph{space projection} $\Pi_{\phi}: C_{T}(\Phi') \rightarrow C_{T}(\R)$ as the mapping $\Pi_{\phi}(x)=\inner{x}{\phi} \defeq (t \mapsto \inner{x(t)}{\phi})$. Likewise, given $t \in [0,T]$, define the \emph{time projection} $\Pi_{t}: C_{T}(\Phi') \rightarrow \Phi'$ as the mapping $\Pi_{t}(x)=x(t)$. For a cylindrical measure $\mu$ on $C_{T}(\Phi)$ we define the Fourier transform $\hat{\mu}_{t}$ at time $t$ as the Fourier transform of $\mu_{t} \defeq \mu \circ \Pi^{-1}_{t}$, that is 
$$ \hat{\mu}_{t}(\phi)= \int_{C_{T}(\Phi')} e^{i \inner{x(t)}{\phi}} \, d \mu, \quad \forall \, \phi \in \Phi. $$
The Fourier transform $\hat{\mu}_{X,t}$ at time $t$ of the cylindrical random variable $X$ in $C_{T}(\Phi')$ is that of its cylindrical measure $\mu_{X}$ at time $t$. 


Let $q$ be a continuous seminorm on $\Phi$. We will need the following \emph{modulus of continuity}:
\begin{enumerate}
\item If $x \in C_{T}(\Phi'_{q})$, $\delta>0$, let 
$$w_{x}(\delta,q)=\sup \{ q'(x(t)-x(s)): s,t \in [0,T], \, \abs{s-t}< \delta \}, $$
\item If $x \in C_{T}(\Phi'_{q})$, $\phi \in \Phi_{q}$, $\delta>0$, let 
$$w_{x}(\delta,\phi)=\sup \{ \abs{\inner{x(t)-x(s)}{\phi}}: s,t \in [0,T], \, \abs{s-t}< \delta \}, $$
\end{enumerate}

In the following theorem a characterization for the compact subsets in $C_{T}(\Phi')$ is established under the assumption that $\Phi$ is nuclear. Its proof follows from the same arguments used in the proof of Theorem 3.5 in \cite{FonsecaMora:Skorokhod} where one can just replace the Skorokhod space $D_{T}(\Phi')$ by the space $C_{T}(\Phi')$ and the modulus of continuity $w'_{x}$ by $w_{x}$ (see Section 3 in \cite{FonsecaMora:Skorokhod} for definitions on the Skorokhod space $D_{T}(\Phi')$ and modulus of continuity $w'_{x}$). 

\begin{theorem}\label{theoCharactCompactSubsets}
Let $\Phi$ be a nuclear space and let $A \subseteq C_{T}(\Phi')$. Consider the following statements:
\begin{enumerate}
\item $A$ is compact in $C_{T}(\Phi')$. 
\item For any $\phi \in \Phi$, the set $\Pi_{\phi}(A)=\{ \inner{x}{\phi}: x \in A\}$ is compact in $C_{T}(\R)$.
\item There exists a continuous Hilbertian seminorm $q$ on $\Phi$ such that $A$ is compact in $C_{T}(\Phi'_{q})$. 
\end{enumerate}
Then we have (1) $\Rightarrow$ (2), (3) $\Rightarrow$ (1),  and if $\Phi$ is also barrelled we have (2) $\Rightarrow$ (3). 
\end{theorem}

In many of our arguments we will require the following properties of the space $C_{T}((\widehat{\Phi}_{\theta})')$ where $\theta$ is a weaker countably Hilbertian topology on $\Phi$.

\begin{proposition} \label{propProperContiSpaceWeakerCHT}
Let $\theta$ be a weaker countably Hilbertian topology on $\Phi$. Then, 
\begin{enumerate}
\item $C_{T}((\widehat{\Phi}_{\theta})')$ is Suslin.
\item $\mathscr{C}(C_{T}((\widehat{\Phi}_{\theta})')) =  \mathcal{B}(C_{T}((\widehat{\Phi}_{\theta})'))$
\item The compact subsets of $C_{T}((\widehat{\Phi}_{\theta})')$ are metrizable. 
\end{enumerate}
\end{proposition}
\begin{proof}
By Proposition 1.6(1) in \cite{Jakubowski:1986}, $C_{T}((\widehat{\Phi}_{\theta})')$ is a closed subspace of the Skorokhod space $D_{T}((\widehat{\Phi}_{\theta})')$ and the Skorokhod topology coincides with the uniform topology. Since $D_{T}((\widehat{\Phi}_{\theta})')$ is Suslin and each compact subset of it is metrizable (see Proposition 3.2 and Lemma 4.2 in \cite{FonsecaMora:Skorokhod}), then $C_{T}((\widehat{\Phi}_{\theta})')$ inherits these properties and thus (1)-(3) follows.  
\end{proof}

We denote by $C_{\infty}(\Phi')$ the collection of all continuous mappings $x:[0,\infty) \rightarrow \Phi'$. We equip $C_{\infty}(\Phi')$ with the topology of \emph{local uniform convergence} (also known as the topology of  \emph{uniform convergence on compact intervals of time}), i.e. if $\Pi$ is any generating family of seminorms for the topology on $\Phi'$, the topology on $C_{\infty}(\Phi')$ is the locally convex topology generated by the family of seminorms
$$ x \mapsto \sup_{t \in [0,T]} p(x(t)), \quad \forall \, p \in \Pi, \, T >0. $$

For every $n \in \N$, let $r_{n}: C_{\infty}(\Phi') \rightarrow C_{n}(\Phi')$, $r_{n}(x)(t)=x(t)$, $0 \leq t \leq n$. Each $r_{n}$ is linear and continuous. Then the topology on $C_{\infty}(\Phi')$ is the projective limit topology with respect to the family $(C_{n}(\Phi'), r_{n} : n \in \N)$. Furthermore, the results in Proposition \ref{propProperContiSpaceWeakerCHT} are also valid for the space $C_{\infty}((\widehat{\Phi}_{\theta})')$.

\section{Proof of our main results}\label{sectProofMainResults}

\begin{assumption}
Unless otherwise indicated, in this section we will always assume that $\Phi$ is a nuclear space.  
\end{assumption}

In this section we give a proof of our main theorems in Section \ref{sectMainTheorems}. We will start with the proof of Theorem \ref{theoUniforTightSpaceContiTimeInfty}, which for  the reader's convenience will be divided in two main steps, the next lemma being the first of these. 

\begin{lemma}\label{lemmaUniforTightSpaceContiOnT}
For every $T>0$ there exists a weaker countably Hilbertian topology $\theta$ on $\Phi$ such that  $(\mu_{\alpha} : \alpha \in A)$ is uniformly tight on $C_{T}((\widehat{\Phi}_{\theta})')$, and in particular is uniformly tight on $C_{T}(\Phi')$.
\end{lemma}
\begin{proof}
Large parts of the proof are analogue to those of Theorem 5.2 in \cite{FonsecaMora:Skorokhod}, there for the Skorokhod space. Below we summarize the proof in five steps and provide details only where the proof requires a change on the arguments. 

{\bf Step 1.} \emph{There exists a weaker countably Hilbertian topology $\theta$ on $\Phi$ such that $\forall \alpha \in A$, $\mu_{\alpha}$ is a Radon probability measure on $C_{T}((\widehat{\Phi}_{\theta})')$. }

In effect, for each $\alpha \in A$ there exists a cylindrical random variable $X_{\alpha} $ in $C_{T}(\Phi')$ such that $\mu_{\alpha}$ is the cylindrical distribution of $X_{\alpha}$. Without loss of generality we can assume all the random variables $X_{\alpha}$ are defined on the same probability space. 

Now, by Lemma 5.1, p.85, in \cite{Kallenberg}, we have
\begin{equation}\label{eqEquiContiCylRVAndFourierTransf}
\Prob (\abs{X_{\alpha}(\phi)(t)} \geq \epsilon) 
\leq \frac{\epsilon}{2} \int^{2/\epsilon}_{-2/\epsilon} (1-\Exp e^{i s X(\phi)(t)}) ds 
= \frac{\epsilon}{2} \int^{2/\epsilon}_{-2/\epsilon} (1-\hat{\mu}_{\alpha,t}(s\phi)) ds
\end{equation}
is valid for every  $\epsilon >0$, $\alpha \in A$, $t \in [0,T]$, and $\phi \in \Phi$. Then the equicontinuity at zero of the family $(\hat{\mu}_{\alpha, t}: t \in [0,T], \alpha \in A)$ shows that the family  $(\phi \mapsto X_{\alpha}(\phi)(t):  t \in [0,T], \alpha \in A)$ of cylindrical random variables in $\Phi'$ is equicontinuous at zero. We can therefore apply the regularization theorem (Theorem 3.2 in \cite{FonsecaMora:Existence}) to show that there exists a weaker countably Hilbertian topology $\theta$ on $\Phi$ such that $\forall \alpha \in A$ there exists a $(\widehat{\Phi}_{\theta})'$-valued continuous process $Y_{\alpha}=(Y_{\alpha,t}: t \geq 0)$ that is a version of $X_{\alpha}$ (unique up to indistinguishable versions). The equicontinuity in \eqref{eqEquiContiCylRVAndFourierTransf} guarantees that $\theta$ does not depend on $\alpha$. 

Then for each $\alpha \in A$ the mapping $\omega \mapsto ( t \mapsto Y_{\alpha,t}(\omega))$ from $\Omega$ into $C_{T}((\widehat{\Phi}_{\theta})')$ is $\mathcal{F}/\mathscr{C}(C_{T}((\widehat{\Phi}_{\theta})'))$-measurable. But since $\mathscr{C}(C_{T}((\widehat{\Phi}_{\theta})')) =  \mathcal{B}(C_{T}((\widehat{\Phi}_{\theta})'))$ (Proposition \ref{propProperContiSpaceWeakerCHT}), then each $Y_{\alpha}$ defines a random variable in $C_{T}((\widehat{\Phi}_{\theta})')$ which has a Radon probabilty distribution since $C_{T}((\widehat{\Phi}_{\theta})')$ is Suslin (Proposition \ref{propProperContiSpaceWeakerCHT}). But by our construction the  probabilty distribution of $Y_{\alpha}$ coincides with $\mu_{\alpha}$, then showing our assertion holds true. 

{\bf Step 2.} \emph{Given $\epsilon>0$, there exists a continuous Hilbertian seminorm $p$ on $\widehat{\Phi}_{\theta}$ (with $\theta$ as in Step 1) such that}
\begin{equation}\label{eqEquicontiFourierTransforms}
\sup_{\alpha \in A} \int_{C_{T}((\widehat{\Phi}_{\theta})')} \, \sup_{t \in [0,T]} \abs{1-e^{i\inner{x(t)}{\phi}}} d\mu_{\alpha} \leq \frac{\epsilon}{12}+2p(\phi)^{2}, \quad \forall \, \phi \in \Phi. 
\end{equation}

In effect \eqref{eqEquicontiFourierTransforms} can be proved from the fact that $\widehat{\Phi}_{\theta}$ is ultrabornological, the conclusion of Step 1, and by using Proposition 5.4 in \cite{FonsecaMora:Skorokhod}. Indeed, the result of Proposition 5.4 in \cite{FonsecaMora:Skorokhod} is formulated for the Skorokhod space on $[0,T]$, but the arguments used in its proof (in particular in the proof of Lemma 5.8 in \cite{FonsecaMora:Skorokhod}) are equally valid for the space of continuous mappings on $[0,T]$ under uniform convergence. We leave to the reader the task of checking this assertion.

{\bf Step 3.} \emph{Given $\epsilon>0$, there exists a continuous Hilbertian seminorm $q$ on $\Phi$, such that $p \leq q$ (with $p$ as in Step 2) and $i_{p,q}$ is Hilbert-Schmidt, and a $C>0$, such that}
\begin{equation}\label{eqUnifTightnessBallQNorm}
 \inf_{\alpha \in A} \mu_{\alpha} \left( x \in C_{T}(\Phi'_{q}): \sup_{t \in [0,T]} q'(x(t)) \leq C \right) \geq 1-\frac{\epsilon}{2}.
\end{equation}

In effect, given $\epsilon >0$ and by using the conclusion of Step 2, we can choose a continuous Hilbertian seminorm $q$ on $\Phi$ with $p \leq q$ and $i_{p,q}$ is Hilbert-Schmidt (note that since $p$ is continuous on $\widehat{\Phi}_{\theta}$ then $p$ is continuous on $\Phi$). Therefore by following similar arguments to those used in the proof of Theorem 5.2 in \cite{FonsecaMora:Skorokhod} we can show that 
for any $C>0$ and any complete orthonormal system $(\phi^{q}_{k})_{k \in \N} \subseteq \Phi$ in $\Phi_{q}$, we have
\begin{multline*}
  \mu_{\alpha} \left( x \in C_{T}((\widehat{\Phi}_{\theta})'): \sup_{t \in [0,T]}  \sum_{k=1}^{\infty} \abs{\inner{x(t)}{\phi_{k}^{q}}}^{2} > C^{2} \right) 
\\ 
\leq \frac{\sqrt{e}}{\sqrt{e}-1} \left(\frac{\epsilon}{12} + \frac{2}{C^{2}} \norm{ i_{p,q} }^{2}_{\mathcal{L}_{2}(\Phi_{q}, \Phi_{p}) }  \right).
\end{multline*}
Then choosing $C$ large enough so that $\displaystyle{\frac{2}{C^{2}} \norm{ i_{p,q} }^{2}_{\mathcal{L}_{2}(\Phi_{q}, \Phi_{p}) } < \frac{\epsilon}{12}}$ and considering the probability of the complement we get \eqref{eqUnifTightnessBallQNorm}.  

{\bf Step 4.} \emph{Given $\epsilon >0$, there exists a continuous Hilbertian seminorm $\varrho$ on $\Phi$, such that $q \leq \varrho$ (with $q$ as in Step 3) and $i_{q,\varrho}$ is Hilbert-Schmidt, and a compact subset $\Gamma$ in $C_{T}(\Phi'_{\varrho})$ such that $\inf_{\alpha \in A} \mu_{\alpha}(\Gamma) \geq 1-\epsilon$. }

In effect, let $\epsilon>0$ and let $q$ and $C$ as in Step 3 and choose a continuous Hilbertian seminorm $\varrho$ on $\Phi$ such that $q \leq \varrho$   and $i_{q,\varrho}$ is Hilbert-Schmidt. 

Let $(\phi^{\varrho}_{k})_{k \in \N} \subseteq \Phi$ be a complete orthonormal system in $\Phi_{\varrho}$. For each $j \in \N$, from our assumption of tightness of the family $(\mu_{\alpha} \circ \Pi^{-1}_{\phi_{j}^{\varrho}}: \alpha \in A)$ and Theorem 7.3 in \cite{Billingsley} there exists a compact subset $B_{j}$ of $C_{T}(\R)$ such that 
\begin{equation}\label{eqMeasureSpaceProjectionsTightness}
\inf_{\alpha \in A} \mu_{\alpha} \circ \Pi^{-1}_{\phi_{j}^{\varrho}}(B_{j})> 1 -\frac{\epsilon}{2^{j+1}},
\end{equation}
and
\begin{equation}\label{eqModulusContinuSpaceProjections}
\lim_{\delta \rightarrow 0} \sup_{x \in \Pi^{-1}_{\phi_{j}^{\varrho}}(B_{j})} w_{x}(\delta, \phi_{j}^{\varrho})=0. 
\end{equation}
Let 
$$\Gamma=\left( \bigcap_{j =1}^{\infty} \Pi_{\phi_{j}^{\varrho}}^{-1} (B_{j}) \right) \cap \left\{ x \in C_{T}(\Phi'_{q}): \sup_{t \in [0,T]} q'(x(t)) \leq C \right\}.$$
Observe that by \eqref{eqUnifTightnessBallQNorm} and \eqref{eqMeasureSpaceProjectionsTightness} we have 
$ \inf_{\alpha \in A} \mu_{\alpha}(\Gamma) \geq 1-\epsilon$. 

We must show that $\Gamma$ is compact in $C_{T}(\Phi'_{\varrho})$. To do this, observe that since the image of $\{ f \in \Phi'_{q}: q'(f) \leq C \}$ under $i'_{q,\varrho}$ is relatively compact  in $\Phi'_{\varrho}$, then by Theorem 2.4.3 in \cite{KallianpurXiong} it suffices to show that $\sup_{x \in \Gamma} w_{x}(\delta,\varrho) \rightarrow 0$ as $\delta \rightarrow 0$.  

In effect, observe that by our definition of $\Gamma$ and since $i_{q, \varrho}$ is Hilbert-Schmidt, for any $\delta >0$ we have
\begin{eqnarray*}
\sum_{j =1}^{\infty}  \sup_{x \in \Gamma}
w_{x}(\delta, \phi_{j}^{\varrho})^{2} 
& \leq  & 4 \sum_{j =1}^{\infty} 
\sup_{x \in \Gamma} \sup_{t \in [0,T]}  \abs{\inner{x(t)}{\phi_{j}^{\varrho}}}^{2}  \\ 
& \leq & 4  \sum_{j =1}^{\infty} \sup_{x \in \Gamma} \sup_{t \in [0,T]}   q'(x(t))^{2}  q(\phi_{j}^{\varrho})^{2}  \\
& \leq & 4 C^{2} \norm{i_{q,\varrho}}^{2}_{\mathcal{L}_{2}(\Phi_{\varrho},\Phi_{q})} < \infty. 
\end{eqnarray*}

Therefore, by \eqref{eqModulusContinuSpaceProjections} and the dominated convergence theorem it follows that
\begin{eqnarray*}
\lim_{\delta \rightarrow 0} \sup_{x \in \Gamma} w_{x}(\delta, \varrho)^{2} 
& = & \lim_{\delta \rightarrow 0} \sup_{x \in \Gamma} \sup_{\abs{s-t}< \delta}   \sum_{j =1}^{\infty} \abs{\inner{x(t)-x(s)}{\phi_{j}^{\varrho}}}^{2}  \\
& \leq & \lim_{\delta \rightarrow 0} \sum_{j =1}^{\infty}  \sup_{x \in \Gamma}
w_{x}(\delta, \phi_{j}^{\varrho})^{2} \\
& = & \sum_{j =1}^{\infty} \lim_{\delta \rightarrow 0} \sup_{x \in \Gamma}
w_{x}(\delta, \phi_{j}^{\varrho})^{2} =0. 
\end{eqnarray*}

{\bf Step 5.} \emph{There exists a weaker countably Hilbertian topology $\vartheta$ on $\Phi$ such that $(\mu_{\alpha}: \alpha \in A)$ is uniformly tight on $C_{T}((\widehat{\Phi}_{\vartheta})')$. }

In effect, choose any  decreasing sequence of positive real numbers $(\epsilon_{n}: n \in \N)$ converging to zero. For each $n \in \N$ by Step 4 there exists a continuous Hilbertian seminorm $\varrho_{n}$ on $\Phi$ and a compact subset $\Gamma_{n}$ in $C_{T}(\Phi'_{\varrho_{n}})$ such that 
$ \inf_{\alpha \in A} \mu_{\alpha} (\Gamma_{n}) \geq 1-\epsilon_{n}.$

Denote by $\vartheta$ the weaker countably Hilbertian topology on $\Phi$ generated by the seminorms $(\varrho_{n}: n \in \N)$. Then each inclusion mapping $i'_{\varrho_{n}}: \Phi'_{\varrho_{n}} \rightarrow (\widehat{\Phi}_{\vartheta})'$ is linear continuous, thus each $\Gamma_{n}$ is compact in $C_{T}((\widehat{\Phi}_{\vartheta})')$. Hence the family
$(\mu_{\alpha}: \alpha \in A)$ is uniformly tight on $C_{T}((\widehat{\Phi}_{\vartheta})')$. 

Finally, the fact that $(\mu_{\alpha}: \alpha \in A)$ is uniformly tight on $C_{T}(\Phi')$ follows from Step 5 and the fact that the canonical inclusion from $(\widehat{\Phi}_{\vartheta})'$ into $\Phi'$ is linear continuous.  
\end{proof}

The second main step in the proof of  Theorem \ref{theoUniforTightSpaceContiTimeInfty} is provided in the next lemma. 

\begin{lemma}\label{lemmaExistenceCompactTightness}
Given $\epsilon >0$, there exists a weaker countably Hilbertian topology $\theta_{\epsilon}$ on $\Phi$ and a compact subset $\mathcal{K}_{\epsilon}$ of $C_{\infty}((\widehat{\Phi}_{\theta_{\epsilon}})')$ such that $\sup_{\alpha \in A} \mu_{\alpha}(\mathcal{K}_{\epsilon}^{c})< \epsilon$.
\end{lemma}
\begin{proof}
First, for each $n \in \N$ by Lemma \ref{lemmaUniforTightSpaceContiOnT} the family $(\mu_{\alpha} \circ r_{n}^{-1}: \alpha \in A)$ is uniformly tight on $C_{n}(\Phi')$. So by Theorem \ref{theoCharactCompactSubsets} given $\epsilon>0$ there exists a continuous Hilbertian seminorm $q_{n}$ on $\Phi$ and $\mathcal{K}_{n} \subseteq C_{n}(\Phi'_{q_{n}})$ compact such that $\sup_{\alpha \in A} \mu_{\alpha} \circ r_{n}^{-1}(\mathcal{K}_{n}^{-1})< \epsilon / 2^{n}$. 

By Theorem 2.4.3 in \cite{KallianpurXiong} it follows that:
\begin{enumerate}[label=(\alph*)]
    \item there exists a compact subset $K_{n}$ of $\Phi'_{q_{n}}$ such that $x(t) \in K_{n}$ for every $ x \in \mathcal{K}_{n}$ and $t \in [0,n]$,
    \item $\lim_{\delta \rightarrow 0} \sup_{x \in \mathcal{K}_{n}} w_{x}(\delta,q)=0$.
\end{enumerate}   

Let $\theta$ be the weaker countably Hilbertian topology on $\Phi$ generated by the family $(q_{n}: n \in \N)$ and let $\mathcal{K}= \bigcap_{n=1}^{\infty} r_{n}^{-1}(\mathcal{K}_{n}) \subseteq C_{\infty}((\widehat{\Phi}_{\theta})')$. Since each $r_{n}$ is continuous, it is clear that $\mathcal{K}$ is closed in $C_{\infty}((\widehat{\Phi}_{\theta})')$, we must show that it is compact. 

Let $K = \bigcap_{n =1}^{\infty} K_{n}$. From the properties of the sets $K_{n}$, we can conclude that the set $K$ is compact in $(\widehat{\Phi}_{\theta})'$ and $x(t) \in K$ $\forall t \geq 0$, $x \in \mathcal{K}$.  Then by the Arzel\`{a}-Ascoli Theorem (see \cite{Kelley}, Theorem 7.18, p.234) in order to show that $\mathcal{K}$ is compact in $C_{\infty}((\widehat{\Phi}_{\theta})')$ it is enough to show that 
$$ \lim_{\delta \rightarrow 0} \sup_{r_{n} (x) \in r_{n}(\mathcal{K})} w'_{x}(n,\delta,p)=0, \quad \forall n \in \N, \, p \in \Pi, $$
where $\Pi$ is a generating family of seminorms on $(\widehat{\Phi}_{\theta})'$ and for $n \in \N$, $\delta >0$ and $p \in \Pi$,
$$ w'_{x}(n,\delta,p) = \sup \{ p(x(t)-x(s)): 0 \leq s \leq t \leq n, \, \abs{s-t}< \delta  \}. $$
Since $(\widehat{\Phi}_{\theta})'$  is equipped with its strong topology, a generating family of seminorms is the collection $\eta_{B}(f)=\sup_{\phi \in B} \abs{ \inner{f}{\phi}}$, where $B$ ranges on the bounded subsets in  $\widehat{\Phi}_{\theta}$. 

Choose any $B \subseteq \widehat{\Phi}_{\theta}$ bounded and $n \in \N$. Since the family $(q_{k}: k \in \N)$ generates the topology $\theta$, then as $B$ is bounded it follows by definition that there exists $C_{n}>0$ such that $B \subseteq C_{n} B_{q_{n}}(1)$. Hence for $f \in \Phi'_{q_{n}}$, 
$$ \eta_{B}(f)=\sup_{\phi \in B} \abs{\inner{f}{\phi}} 
\leq \sup_{\phi \in  C_{n} B_{q_{n}}(1)} \abs{\inner{f}{\phi}} 
= C_{n} \sup_{\phi \in  B_{q_{n}}(1)} \abs{\inner{f}{\phi}}
= C_{n} q'_{n}(f). $$

Therefore, 
$$ \lim_{\delta \rightarrow 0} \sup_{r_{n} (x) \in r_{n}(\mathcal{K})} w'_{x}(n,\delta, \eta_{B}) \leq C_{n} \lim_{\delta \rightarrow 0} \sup_{x \in \mathcal{K}_{n}} w_{x}(\delta,q_{n}) = 0. $$
So we have proved that $\mathcal{K}$ is compact.
\end{proof}

\begin{proof}[Proof of  Theorem \ref{theoUniforTightSpaceContiTimeInfty}]
Let $(\epsilon_{m}: m \in \N)$ be a decreasing sequence of positive numbers converging to $0$. Then $\forall m \in \N$, $\exists \theta_{m}$ and $\mathcal{K}_{m}$ as in Lemma \ref{lemmaExistenceCompactTightness}. Let $\theta$ the weaker countably Hilbertian topology on $\Phi$ generated by the Hilbertian seminorms generating the topologies $\theta_{m}$, $\forall m \in \N$. Then each $\mathcal{K}_{m}$ is compact in $C_{\infty}((\widehat{\Phi}_{\theta})')$ and therefore the family $(\mu_{\alpha}: \alpha \in A)$ is uniformly tight on $C_{\infty}((\widehat{\Phi}_{\theta})')$, hence on $C_{\infty}(\Phi')$. 
\end{proof}

\begin{remark}\label{remaNecessaryCondTighness}
If $\Phi$ is a barrelled nuclear space and the family $(\mu_{\alpha}: \alpha \in A)$  is uniformly tight on $C_{\infty}(\Phi')$, then conditions \ref{condEquiFourTrans} and \ref{condFinDimenTight} in Theorem \ref{theoUniforTightSpaceContiTimeInfty} are satisfied. So this see, observe that by Proposition 1.6.vi) in \cite{Jakubowski:1986} the family of probability measures $(\mu_{\alpha}\circ \Pi^{-1}_{t}: \alpha \in A, t \geq 0)$ is uniformly tight on $\Phi'$. But because $\Phi$ is barrelled and nuclear, it follows that for each $T>0$ the family of Fourier transforms $(\hat{\mu}_{\alpha,t}: t \in [0,T], \alpha \in A)$ is equicontinuous at zero (see \cite{DaleckyFomin}, Theorem III.2.7, p.104). Finally for each $\phi \in \Phi$ the continuity of the space projection map $\Pi_{\phi}$ and that  $(\mu_{\alpha}: \alpha \in A)$ is uniformly tight on $C_{\infty}(\Phi')$ implies that  $(\mu_{\alpha} \circ \Pi_{\phi}^{-1}: \alpha \in A)$ is uniformly tight on $C_{\infty}(\R)$.    
\end{remark}

\begin{remark}\label{remaUltraboTheoTight} As mentioned in Section \ref{sectMainTheorems}, if $\Phi$ is an ultrabornological nuclear space, then condition \ref{condFinDimenTight} implies condition \ref{condEquiFourTrans}  in Theorem \ref{theoUniforTightSpaceContiTimeInfty}. 
This conclusion follows from Proposition 5.4 in \cite{FonsecaMora:Skorokhod}, which is valid for the space $C_{T}(\Phi')$, this as  pointed out in Step 2 in the proof of Lemma \ref{lemmaUniforTightSpaceContiOnT}. 
\end{remark}

\begin{proof}[Proof of Theorem \ref{theoWeakConvergCylinProcesInCTInfy}] First, by (1) for each $n \in \N$ the cylindrical process   $X^{n}=(X^{n}_{t}: t \geq 0)$ defines a cylindrical random variable in $C_{\infty}(\Phi')$. Let $\mu_{n}$ be the cylindrical distribution of $X^{n}$. By (2) for each $T>0$ the family of Fourier transforms $(\hat{\mu}_{n,t}: t \in [0,T], n \in \N)$ is equicontinuous at the origin.  As shown in Step 1 in the proof of Lemma \ref{lemmaUniforTightSpaceContiOnT}, each $X^{n}$ extends to a genuine random variable $Y^{n}$ in $C_{\infty}(\Phi')$ and $\mu_{n}$ extends to a probability measure on $C_{\infty}(\Phi')$ that corresponds to the probability distribution of $Y^{n}$. By (3) and Theorem \ref{theoUniforTightSpaceContiTimeInfty} there exists a weaker countably Hilbertian topology $\theta$ on $\Phi$ such that  $(\mu_{n}: n \in \N)$ is uniformly tight on $C_{\infty}((\widehat{\Phi}_{\theta})')$. Then one can show from (4) that every subsequence of $(\mu_{n}: n \in \N)$ contains a further subsequence that converges weakly to the same limit $\mu$ in $\goth{M}^{1}(C_{\infty}((\widehat{\Phi}_{\theta})'))$. Hence, Theorem 2 in \cite{Billingsley} shows that  
$\mu_{n} \Rightarrow \mu$ in $\goth{M}^{1}(C_{\infty}((\widehat{\Phi}_{\theta})'))$.
Finally because the inclusion $j_{\theta}$ from $(\widehat{\Phi}_{\theta})'$ into $\Phi'$ is linear and continuous, then for every continuous bounded real-valued function $f$ defined on  $\Phi'$ we have $f \circ j_{\theta}$ is a continuous bounded real-valued function defined on $(\widehat{\Phi}_{\theta})'$. Therefore, the fact that $\mu_{n} \Rightarrow \mu$ in $\goth{M}^{1}(C_{\infty}((\widehat{\Phi}_{\theta})'))$ implies that $\mu_{n} \Rightarrow \mu$ in $\goth{M}^{1}(C_{\infty}(\Phi'))$.
 Hence each $Y^{n}$ is a $C_{\infty}((\widehat{\Phi}_{\theta})')$-valued random variable and $Y$ is a $C_{\infty}((\widehat{\Phi}_{\theta})')$-valued random variable whose probability distribution is $\mu$. Therefore (i)-(iii) are clearly satisfied. 
\end{proof}

\begin{proof}[Proof of Theorem \ref{theoWeakConvUltrabornological}]
If $\Phi$ is an ultrabornological nuclear space, then condition \ref{condTighTheoWeakConv} implies condition \ref{condEquiTheoWeakConv}  in Theorem \ref{theoWeakConvergCylinProcesInCTInfy}. In effect, as in Remark \ref{remaUltraboTheoTight} we have condition \ref{condEquiFourTrans}  in Theorem \ref{theoUniforTightSpaceContiTimeInfty} holds true. Then condition \ref{condEquiTheoWeakConv}  in Theorem \ref{theoWeakConvergCylinProcesInCTInfy} follows by \eqref{eqEquiContiCylRVAndFourierTransf}. 

Now, if we apply Theorem \ref{theoWeakConvergCylinProcesInCTInfy} to the sequence $(X^{n}:n \in \N)$ of $\Phi'$-valued regular continuous processes we have by \ref{concluRegularVersTheoWeakConv} in Theorem \ref{theoWeakConvergCylinProcesInCTInfy} that each $X^{n}$ defines a $C_{\infty}(\Phi')$-valued random variable, by \ref{concluTightTheoWeakConv} in Theorem \ref{theoWeakConvergCylinProcesInCTInfy} that the sequence $(X^{n}: n \in \N)$ is uniformly tight on $C_{\infty}(\Phi')$, and by \ref{concluConvergenTheoWeakConv} in Theorem \ref{theoWeakConvergCylinProcesInCTInfy} that there exists a 
$C_{\infty}(\Phi')$-valued random variable $X$ (in particular, $X$ is a $\Phi'$-valued regular continuous process) such that $X^{n} \Rightarrow X$ in $C_{\infty}(\Phi')$.  
\end{proof}

\section{Application: A central limit theorem for local martingales}\label{sectCentralLimitTheorem}

Now we illustrate the usefulness of Theorem \ref{theoWeakConvUltrabornological} by applying it to the study of weak convergence of a sequence of continuous local martingales taking values in the dual of an ultrabornological nuclear space $\Phi$. 

Recall that a $\Phi'$-valued process $M=(M_{t}: t \geq 0)$ is said to be a \emph{(local) martingale}  if for each $\phi \in \Phi$, we have $M(\phi) \defeq (\inner{M_{t}}{\phi}: t \geq 0)$ is a (local) martingale. For each $\phi \in \Phi$, we denote by $\langle M(\phi) \rangle$ the quadratic variation of $M(\phi)$. 

\begin{theorem}\label{theoWeakConvContLocalMartin}
Let $\Phi$ be an ultrabornological nuclear space. Assume for each $n=1,2,\dots$, $M^{n}=(M^{n}_{t}: t \geq 0)$ is a $\Phi'$-valued regular local martingale with continuous paths  satisfying $M^{n}_{0}=0$. Moreover, assume the following:
\begin{enumerate}
\item \label{assuTighQuaVaria} For every $\phi \in \Phi$,  $(\langle M^{n}(\phi) \rangle: n \in \N)$ is uniformly tight in $C_{\infty}(\R)$. 
\item $\forall$ $m \in \N$, $\phi_{1}, \dots, \phi_{m} \in \Phi$, $t_{1}, \dots, t_{m} \geq 0$, the distribution of the $m$-dimensional random vector $(M_{t_{1}}^{n}(\phi_{1}), \cdots, M_{t_{m}}^{n}(\phi_{m}))$ converges in distribution to some probability measure on $\R^{m}$. 
\end{enumerate}
Then, $(M^{n}:n \in \N)$ is  uniformly tight on $C_{\infty}(\Phi')$ and there exists a $\Phi'$-valued regular local martingale $M=(M_{t}: t \geq 0)$ with continuous paths and $M_{0}=0$, such that $M^{n} \Rightarrow M$ in $C_{\infty}(\Phi')$. 
\end{theorem}
\begin{proof}
Let $\phi \in \Phi$. By \ref{assuTighQuaVaria} and Lemma III.11 in \cite{Rebolledo:1980} (see also Theorem A in \cite{Nakao:1986}) the sequence $(M^{n}(\phi): n \in \N)$ is uniformly tight in $C_{\infty}(\R)$. Then, by  Theorem \ref{theoWeakConvUltrabornological} there exists a $\Phi'$-valued regular process $M=(M_{t}: t \geq 0)$ with continuous paths such that $M^{n} \Rightarrow M$ in $C_{\infty}(\Phi')$. Moreover, because the property of being a real-valued  local martingale with continuous paths starting at zero is preserved under weak convergence in $C_{\infty}(\R)$ (see Theorem B in \cite{Nakao:1986}), for each $\phi \in \Phi$ we have $M(\phi)$ is a local martigale with $M_{0}(\phi)=0$. That is, $M$ is a $\Phi'$-valued local martingale with $M_{0}=0$. 
\end{proof}

A central limit theorem for local martingales taking values in the dual of a nuclear Fr\'echet space was proved by Fierro in \cite{Fierro:1988}. We will apply Theorem \ref{theoWeakConvContLocalMartin} to prove an extension of the central limit theorem to the case of local martingales taking values in the dual of an ultrabornological nuclear space.  Recall that a $\Phi'$-valued process $M=(M_{t}: t \geq 0)$ is called \emph{Gaussian} if for any $m \in \N$ and any $\phi_{1}, \dots, \phi_{m} \in \Phi$,  $\left((M_{t}(\phi_{1}), \dots,M_{t}(\phi_{n})):  t \geq 0 \right)$ is a Gaussian process in $\R^{m}$. 

\begin{theorem}[Central limit theorem]\label{theoCentralLimitTheorem}
Let $\Phi$ be an ultrabornological nuclear space. Assume for each $n=1,2,\dots$, $M^{n}=(M^{n}_{t}: t \geq 0)$ is a $\Phi'$-valued regular local martingale with continuous paths  satisfying $M^{n}_{0}=0$. Assume the existence of a mapping $A: [0,\infty) \times \Phi \rightarrow \R$ such that:
\begin{enumerate}
\item $\forall t \geq 0$, $A(t,\cdot)$ is a continuous positive semidefinite quadratic form on $\Phi$, 
\item $\forall \phi \in \Phi$, $A(\cdot, \phi)$ is an increasing continuous function with $A(0,\phi)=0$. 
\item $\forall t \geq 0$ and $\phi \in \Phi$,  $\langle M^{n}(\phi) \rangle_{t}$ converges in probability towards $A(t,\phi)$ as $n \rightarrow \infty$. 
\end{enumerate}
Then, there exists a $\Phi'$-valued Gaussian mean-zero continuous martingale $M=(M_{t}: t \geq 0)$ with covariance $\Exp \left[ \abs{M_{t}(\phi)}^{2} \right] = A(t,\phi)$ $\forall t \geq 0$ and $\phi \in \Phi$, such that $M^{n} \Rightarrow M$ in $C_{\infty}(\Phi')$. 
\end{theorem}
\begin{proof}
We must check that the conditions in Theorem \ref{theoWeakConvContLocalMartin} are satisfied. 

Let $\phi \in \Phi$. Since the processes $\langle M^{n}(\phi) \rangle$ are a.e. increasing and the limit $A(\cdot, \phi)$ is an increasing continuous function with $A(0,\phi)=0$, we have  $\langle M^{n}(\phi) \rangle$ converges in probability towards $A(t,\phi)$ uniformly on compact intervals of time as $n \rightarrow \infty$. Then  $(\langle M^{n}(\phi) \rangle: n \in \N)$ is uniformly tight in $C_{\infty}(\R)$. 

Let  $m \in \N$ and $\phi_{1}, \dots, \phi_{m} \in \Phi$. By the central limit theorem for local martingales (see \cite{Rebolledo:1980}) there exists a $\R^{m}$-valued Gaussian mean-zero continuous martingale $M^{\phi_{1}, \dots, \phi_{m}}$ with covariance
$$ \Exp \left[ \norm{M^{\phi_{1}, \dots, \phi_{m}}_{t}}^{2} \right] = \sum_{j=1}^{m} A(t, \phi_{j}), $$
and such that $(M(\phi_{1}), \dots, M(\phi_{m})) \Rightarrow M^{\phi_{1}, \dots, \phi_{m}}$  in $C_{\infty}(\R^{m})$. Therefore, by Theorem  \ref{theoWeakConvContLocalMartin} there exists a $\Phi'$-valued regular local martingale $M=(M_{t}: t \geq 0)$ with continuous paths and $M_{0}=0$, such that $M^{n} \Rightarrow M$ in $C_{\infty}(\Phi')$. By uniqueness of weak limits, we have for  $m \in \N$ and $\phi_{1}, \dots, \phi_{m} \in \Phi$ that $M^{\phi_{1}, \dots, \phi_{m}}$ is an indistinguishable version of $(M(\phi_{1}), \dots, M(\phi_{m}))$. Therefore, $M$ is a Gaussian process and $\Exp \left[ \abs{M_{t}(\phi)}^{2} \right] = A(t,\phi)$ $\forall t \geq 0$ and $\phi \in \Phi$. 
\end{proof}

\section{Application: Weak convergence of solutions to SPDEs with semimartingale noise}\label{secWeakSemimart}

Let $\Phi$ denotes a complete bornological (hence ultrabornological) nuclear space whose strong dual $\Phi'$ is complete and nuclear.  In this section we apply Theorem \ref{theoWeakConvUltrabornological} to formulate sufficient conditions for the weak convergence in $C_{\infty}(\Phi')$ of a sequence of solutions to $\Phi'$-valued stochastic evolution equations with semimartingale noise.

\subsection{Existence and uniqueness of solutions to SPDEs with semimartingale noise}

We denote by $S^{0}$ the space of all real-valued semimartingales equipped with Emery's topology (see Section 12.4 in \cite{CohenElliott}). Further, we denote by $\mathcal{H}^{1}_{S}$ the Banach space (see Section 16.2 in \cite{CohenElliott}) of all the  real-valued semimartingales $x=(x_{t}:t \geq 0)$ satisfying 
$$\norm{x}_{\mathcal{H}^{1}_{S}} \defeq \inf \left\{ \Exp  [m,m]_{\infty}^{1/2} + \Exp \int_{0}^{\infty} \abs{d a_{s} } : x=m+a \right\}<\infty, $$
where the infimum is taken over all the decompositions $x=m+a$ as a sum of a local martingale $m$ and a process of finite variation $a$, and  $([m,m]_{t}: t \geq 0)$ denotes the quadratic variation process associated to the local martingale $m$. 

Assume that the $\Phi'$-valued process $Z=(Z_{t}: t \geq 0)$ is a \emph{semimartingale}, i.e. $\inner{Z}{\phi} \in S^{0}$ for every $\phi \in \Phi$. 
Assume further that the mapping $Z: \Phi \rightarrow S^{0}$ is continuous  and that  $Z$ has continuous paths.

Let $(S(t):t \geq 0) \subseteq \mathcal{L}(\Phi,\Phi)$ be a strongly continuous $C_{0}$-semigroup  with infinitesimal generator $A$  (see \cite{Komura:1968}). Consider the following linear stochastic evolution equations
\begin{equation*}\label{eqLinearSEE}
 dY_{t}=A'Y_{t} dt+dZ_{t}, \quad t \geq 0,
\end{equation*}
with initial condition $Y_{0}=\eta$ $\Prob$-a.e. where $\eta$  is a $\mathcal{F}_{0}$-measurable regular random variable. A $\Phi'$-valued regular adapted process $Y=(Y_{t}: t \geq 0)$ is called a \emph{weak solution} to the above equation if for any given $t > 0$ and  $\phi \in \mbox{Dom}(A)$ we have $\int_{0}^{t} \abs{\inner{Y_{r}}{A\phi}} dr < \infty$ $\Prob$-a.e. and 
\begin{equation*}\label{eqWeakSoluStochEvolEqua}
\inner{Y_{t}}{\phi}= \inner{\eta}{\phi} + \int_{0}^{t} \inner{Y_{r}}{A\phi} dr + \inner{Z_{t}}{\phi}. 
\end{equation*}
If we assume that  $Z$ is a $\mathcal{H}^{1}_{S}$-\emph{semimartingale}, i.e. $\inner{Z}{\phi} \in \mathcal{H}^{1}_{S}$ for every $\phi \in \Phi$, we have by Theorem 6.11 in \cite{FonsecaMora:StochInteg} that there exists a unique weak solution $(Y_{t}: t \geq 0)$ which is regular and has continuous paths, and satisfies $\Prob$-a.e. $\forall \, t \geq 0, \, \phi \in \Phi$,
\begin{equation*}\label{eqExplicitSolSEE}
\inner{Y_{t}}{\phi}= \inner{\eta -Z_{0}}{S(t)\phi}  +\int_{0}^{t} \inner{Z_{s}}{S(t-s)A\phi} ds +\inner{Z_{t}}{\phi}.  
\end{equation*}
By a standard localization argument one can show that the result above remain true if we only assume that $Z$ is \emph{locally a $\mathcal{H}^{1}_{S}$-semimartingale}, i.e. that there exists a sequence of stopping times $(\tau_{n}: n \in \N)$ increasing to $\infty$ $\Prob$-a.e.  such that the stopped process $Z^{\tau_{n}} \defeq Z_{\cdot \wedge \tau_{n}}$ is a $\mathcal{H}^{1}_{S}$-semimartingale.

\subsection{Sufficient conditions for weak convergence of solutions}

For $n=0,1,2, \cdots$, assume that the $\Phi'$-valued process $Z^n=(Z^{n}_{t}: t \geq 0)$ is locally a $\mathcal{H}^{1}_{S}$- semimartingale. We further assume that the mapping $Z^{n}: \Phi \rightarrow S^{0}$ is continuous, and that  $Z^{n}(\phi)$ has continuous paths for each $\phi \in \Phi$. 
Let  $(S^{n}(t):t \geq 0) \subseteq \mathcal{L}(\Phi,\Phi)$ be  strongly continuous $C_{0}$-semigroup with infinitesimal generator $A^{n} \in  \mathcal{L}(\Phi,\Phi)$ and let $\eta^{n}$ be a  $\mathcal{F}_{0}$-measurable regular random variable. 

For every $n$, let $Y^{n}=(Y^{n}_{t}: t \geq 0)$ be the unique weak solution to the initial value problem:
\begin{equation}\label{eqSEEInitialCondition}
dY^{n}_{t}=(A^{n})'Y^{n}_{t} dt +dZ^{n}_{t}, \quad Y^{n}_{0}=\eta^n. 
\end{equation}
As explained above, $(Y^{n}_{t}: t \geq 0)$  is regular and has continuous paths, and satisfies $\Prob$-a.e. $\forall \, t \geq 0$,  $\phi \in \Phi$,
\begin{equation*}\label{eqUniqueWeakSolSEE}
\inner{Y^{n}_{t}}{\phi}= \inner{\eta^{n} -Z^{n}_{0}}{S^{n}(t)\phi}  +\int_{0}^{t} \inner{Z^{n}_{s}}{S^{n}(t-s)A^{n}\phi} ds +\inner{Z^{n}_{t}}{\phi}.
\end{equation*}

In the next theorem we formulate sufficient conditions for the weak convergence of $Y^{n}$ to $Y^{0}$  in $C_{\infty}(\Phi')$. 

\begin{theorem}\label{theoWeakConvSoluSEE}
Assume the following:
\begin{enumerate}
\item \label{assuIndependence} For $n=0,1,2,\dots$, $\eta^{n}-Z^{n}_{0}$ is independent of $Z^{n}$. 
\item \label{assuWeakConvInitCondi} $\eta^{n}-Z^{n}_{0} \Rightarrow \eta^{0}-Z^{0}_{0}$.
\item \label{assuConvSemig} $\forall \phi \in \Phi$, $S^{n}(t)\phi \rightarrow S^{0}(t)\phi$ and $S^{n}(t)A^{n} \phi \rightarrow S^{0}(t) A^{0}\phi$ as $n \rightarrow \infty$ uniformly in $t$ on any bounded interval of time. 
\item  \label{assuConvSemimar} $Z^{n} \Rightarrow Z^{0}$ in $C_{\infty}(\Phi')$.
\end{enumerate}
Then the sequence $(Y^{n}:n \in \N)$ is uniformly tight  in $C_{\infty}(\Phi')$ and $Y^{n} \Rightarrow Y^{0}$ in $C_{\infty}(\Phi')$.
\end{theorem}
\begin{proof} In view of Theorem \ref{theoWeakConvUltrabornological}, we must show the following:
\begin{enumerate}[label=(\alph*)]
\item $\forall \phi \in \Phi$, the sequence of distributions of $\inner{Y^{n}}{\psi}$ is uniformly tight on $C_{\infty}(\R)$.  
\item $\forall m\in \N$, $\phi_{1}, \dots, \phi_{m} \in \Phi$, and $t_{1}, \dots, t_{m} \geq 0$, the
distribution of the $m$-dimensional vector $(\inner{Y^{n}_{t_{1}}}{\phi_{1}}, \dots, \inner{Y^{n}_{t_{m}}}{\phi_{m}})$ converges in distribution to some
probability measure on $\R^{m}$.
\end{enumerate} 
To prove the above, we modify the arguments used by  Fern\'{a}ndez and Gorostiza in \cite{FernandezGorostiza:1992} (Theorem 1). 
For each $n=1,2, \dots$, let $\tilde{Y}^{n}$ be defined by
\begin{equation*} \label{eqModificactionSoluEEE}
\tilde{Y}^{n}_{t}=S^{n}(t)'(\eta^{n}-Z^{n}_{0})+\int_{0}^{t} (A^{0})' S^{0}(t-s)' Z^{n}_{s} ds + Z^{n}_{t}, \quad \forall \, t \geq 0. 
\end{equation*} 
where 
$$ \inner{\int_{0}^{t} (A^{0})' S^{0}(t-s)' x_{s} \, ds}{\phi}= \int_{0}^{t} \inner{x_{s}}{S^{0}(t-s)A^{0} \phi} ds, \quad \forall \, \phi \in \Phi, \, x_{s} \in C_{\infty}(\Phi').$$ 
Observe that in order to prove (a) and (b) it is enough to show 
\begin{enumerate} [label=(\alph*)]
\setcounter{enumi}{2}
\item $\tilde{Y}^{n} \Rightarrow Y^{0}$ in $C_{\infty}(\Phi')$ as $n \rightarrow \infty$, 
\item $\int_{0}^{t} \, \inner{Z^{n}_{s}}{(S^{n}(t-s)A^{n}\phi-S^{0}(t-s)A^{0}\phi)} \, ds \rightarrow 0$ in probability as $n \rightarrow \infty$ for each $t \geq 0$ and $\phi \in \Phi$,     
\item $\forall \phi \in \Phi$, $T>0$, $\epsilon >0$, 
$\displaystyle \Prob \left( \sup_{0 \leq t \leq T} \abs{ \inner{\tilde{Y}^{n}_{t}-Y^{n}_{t}}{\phi}} > \epsilon \right) \rightarrow 0$ as $ n \rightarrow \infty.$
\end{enumerate}
In effect, assume that (c), (d) and (e) holds. We show that (a) holds true. In effect, given $\phi \in \Phi$, by (e) we have $ \inner{\tilde{Y}^{n}-Y^{n}}{\phi} \Rightarrow 0 $
 in $C_{\infty}(\R)$. By (c) we have $\inner{\tilde{Y}^{n}}{\phi} \Rightarrow \inner{Y^{n}}{\phi} $  in $C_{\infty}(\R)$, which implies (a). 
 
 To prove (b), our assumption $Z^{n} \Rightarrow Z^{0}$ in $C_{\infty}(\Phi')$ together with (d) shows that $\forall m\in \N$, $\phi_{1}, \dots, \phi_{m} \in \Phi$, and $t_{1}, \dots, t_{m} \geq 0$, as $n \rightarrow \infty$, 
$$\left(\inner{\tilde{Y}^{n}_{t_{1}}-Y^{n}_{t_{1}}}{\phi_{1}}, \dots, \inner{\tilde{Y}^{n}_{t_{m}}-Y^{n}_{t_{m}}}{\phi_{m}}\right) \Rightarrow (0, \dots, 0). $$ 
Hence from (c) we conclude that (b) holds true. From the arguments above, to prove Theorem \ref{theoWeakConvSoluSEE} we must show (c), (d) and (e). 

Proof of (c):
Let $F:C_{\infty}(\Phi') \rightarrow C_{\infty}(\Phi')$ be defined by 
\begin{equation}\label{eqDefiMappingF}
F(x)(t)= x_{t}+\int_{0}^{t} (A^{0})' S^{0}(t-s)' x_{s} \, ds, 
\end{equation}

We will show that the mapping $F$ is continuous. 
In effect, assume $(x^{\alpha}: \alpha \in A)$ converges to $x$ in $C_{\infty}(\Phi')$. Let $T>0$ and $B \subseteq \Phi$ bounded. Observe that 
\begin{flalign*}
& \sup_{0 \leq t \leq T} \sup_{\phi \in B} \abs{ \inner{\int_{0}^{t} (A^{0})' S^{0}(t-s)' (x^{\alpha}_{s} - x_{s}) \, ds}{\phi}}  
\\
& \leq \sup_{0 \leq t \leq T} \sup_{\phi \in B}  \int_{0}^{t} \inner{x^{\alpha}_{s} - x_{s}}{S^{0}(t-s)A^{0} \phi} ds \\
& \leq T \sup_{0 \leq t \leq T} \sup_{\psi \in K} \abs{\inner{x^{\alpha}_{s} - x_{s}}{\psi}} \rightarrow 0, 
\end{flalign*}
where $K=\bigcup_{\phi \in B} \bigcup_{0 \leq s \leq t \leq T} S^{0}(t-s) A^{0} \phi$ is bounded by the strong continuity of $S^{0}$. Since strong topology on $\Phi'$ is generated by the collection of the seminorms $q_{B}(f)=\sup_{\phi \in B} \abs{\inner{f}{\phi}}$, where $B$ runs over the bounded subsets of $\Phi$, the above calculation shows that $F$ defined by \eqref{eqDefiMappingF} is continuous. 

Observe that from the definition of the mapping $F$ we have
$Y^{0}_{t}=S^{0}(t)'(\eta^{0}-Z^{0}_{0})+F(Z^{0})(t)$ and for each $n=1,2,\dots$ we have $\tilde{Y}^{n}_{t}=S^{n}(t)'(\eta^{n}-Z^{n}_{0})+F(Z^{n})(t)$. Moreover, by our assumption \ref{assuIndependence} for $n=0,1,2,\dots$  we have that $S^{n}(\cdot)'(\eta^{n}-Z_{0}^{n})$ and $F(Z^{n})$ are independent. Hence, to prove (c) it suffices to show  
$S^{n}(\cdot)'(\eta^{n}-Z_{0}^{n}) \Rightarrow S^{0}(\cdot)'(\eta^{0}-Z_{0}^{0})$ in $C_{\infty}(\Phi')$ and $F(Z^{n}) \Rightarrow F(Z^{0})$ in $C_{\infty}(\Phi')$. 
Observe however that by the continuity of $F$ an our assumption $Z^{n} \Rightarrow Z^{0}$ in $C_{\infty}(\Phi')$ we have $F(Z^{n}) \Rightarrow F(Z^{0})$ in $C_{\infty}(\Phi')$. 

Thus, for our final step in the proof of (c) we show $S^{n}(\cdot)'(\eta^{n}-Z_{0}^{n}) \Rightarrow S^{0}(\cdot)'(\eta^{0}-Z_{0}^{0})$ in $C_{\infty}(\Phi')$. 
In effect, by our assumptions \ref{assuWeakConvInitCondi} and \ref{assuConvSemig},  $\forall m\in \N$, $\phi_{1}, \dots, \phi_{m} \in \Phi$, and $t_{1}, \dots, t_{m} \geq 0$, as $n \rightarrow \infty$, 
\begin{flalign*}
& \left(\inner{S^{n}(t_{1})'(\eta^{n}-Z_{0}^{n})}{\phi_{1}}, \dots, \inner{S^{n}(t_{m})'(\eta^{n}-Z_{0}^{n})}{\phi_{m}}\right) \\ 
& =  \left(\inner{\eta^{n}-Z_{0}^{n}}{S^{n}(t_{1}) \phi_{1}}, \dots, \inner{\eta^{n}-Z_{0}^{n}}{S^{n}(t_{m})\phi_{m}}\right) \\
& \hspace{15pt} \Rightarrow \left(\inner{\eta^{n}-Z_{0}^{n}}{S^{0}(t_{1}) \phi_{1}}, \dots, \inner{\eta^{n}-Z_{0}^{n}}{S^{0}(t_{m})\phi_{m}}\right)  \\
& \hspace{30pt} = \left(\inner{S^{0}(t_{1})'(\eta^{n}-Z_{0}^{n})}{\phi_{1}}, \dots, \inner{S^{0}(t_{m})'(\eta^{n}-Z_{0}^{n})}{\phi_{m}}\right)
\end{flalign*}
Similarly, for every $\phi \in \Phi$ we have $\inner{S^{n}(\cdot)'(\eta^{n}-Z_{0}^{n})}{\phi} \Rightarrow \inner{S^{0}(\cdot)'(\eta^{n}-Z_{0}^{n})}{\phi}$ in $C_{\infty}(\R)$. 
Thus by Prokhorov's theorem $( \inner{S^{n}(\cdot)'(\eta^{n}-Z_{0}^{n})}{\phi}: n \in \N)$ is uniformly tight on $C_{\infty}(R)$ for every $\phi \in \Phi$. Then, by Theorem  
\ref{theoWeakConvUltrabornological} we have  $S^{n}(\cdot)'(\eta^{n}-Z_{0}^{n}) \Rightarrow S^{0}(\cdot)'(\eta^{0}-Z_{0}^{0})$ in $C_{\infty}(\Phi')$.

Proof of (d): Since $Z^{n} \Rightarrow Z^{0}$ in $C_{\infty}(\Phi')$, for each $\phi \in \Phi$ we have $\inner{Z^{n}}{\phi}  \Rightarrow \inner{Z^{0}}{\phi}$ in $C_{\infty}(\Phi')$. Then by Prokhorov's theorem $(\inner{Z^{n}}{\phi} : n \in \N)$ is uniformly tight on $C_{\infty}(\R)$. Then, one can show (see the arguments used in the proof of Step 3 in Lemma \ref{lemmaUniforTightSpaceContiOnT}) that for given $t>0$ and $\epsilon>0$ there exist a continuous Hilbertian seminorm $q$ on $\Phi$ and $C>0$ such that 
$$ \sup_{n \in \N} \Prob \left( \sup_{0 \leq s \leq t} q'(Z^{n}_{s}) > C \right) < \epsilon. $$ 

Then
\begin{flalign*}
& \Prob \left( \abs{ \int_{0}^{t} \,\inner{Z^{n}_{s}}{(S^{n}(t-s)A^{n}\phi-S^{0}(t-s)A^{0}\phi)} \, ds } > \epsilon \right) \\
& \leq \Prob \left( t \, \sup_{0 \leq s \leq t} q'(Z^{n}_{s}) \sup_{0 \leq s \leq t} q \left( S^{n}(t-s)A^{n} \phi - S^{0}(t-s)A^{0}\phi \right) > \epsilon \right)\\
& \leq \epsilon + \Prob \left( \sup_{0 \leq s \leq t} q \left( S^{n}(t-s)A^{n} \phi - S^{0}(t-s)A^{0}\phi \right) > \epsilon/Ct \right). 
\end{flalign*}
By \ref{assuConvSemig}, the last term in the above inequality tends to $0$ as $n \rightarrow \infty$. Since $\epsilon$ is arbitrary, we have shown (d).  
  
Proof of (e): Given $t>0$ and $\phi \in \Phi$, we have
$$ \inner{\tilde{Y}^{n}_{t}-Y^{n}_{t}}{\phi}=  \int_{0}^{t} \,\inner{Z^{n}_{s}}{(S^{n}(t-s)A^{n}\phi-S^{0}(t-s)A^{0}\phi)} \, ds, $$
Hence (e) can be proved from similar arguments to those used for (d). 
\end{proof}

\section{Example of Theorems \ref{theoCentralLimitTheorem} and \ref{theoWeakConvSoluSEE}}\label{sectExample}

In this section we consider a concrete example that illustrates the applicability of our results. We benefit from ideas taken from \cite{Fierro:1988, Fouque:1984}. 

\begin{example} 
Let $(B^{i}:i \in \N)$ be a sequence of independent $d$-dimensional Brownian motions starting at zero. For every $n =1,2,\dots$ define a $\mathscr{S}'(\R^{d})$-valued process  $M^{n}=(M^{n}_{t}: t \geq 0)$ by  
$$M^{n}_{t}(\phi)=n^{-1/2} \sum_{i=1}^{n} \int_{0}^{t} \nabla \phi (B^{i}_{s}) \cdot  dB^{i}_{s}, \quad \forall \, t \geq 0, \, \phi \in \mathscr{S}(\R^{d}).$$
Each $M^{n}$ has continuous paths and is \emph{locally a $\mathcal{M}^{2}_{\infty}$-martingale}, i.e.  there exists a sequence of stopping times $(\tau_{k,n}: k \in \N)$ increasing to $\infty$ $\Prob$-a.e.  such that for every $\phi  \in \mathscr{S}(\R^{d})$ we have $M^{n}_{\cdot \wedge \tau_{k,n}}(\phi) \in \mathcal{M}^{2}_{\infty}$ (is a square integrable martingale). Observe moreover that 
$$  \langle M^{n}(\phi) \rangle_{t} = n^{-1} \sum_{i=1}^{n} \int_{0}^{t} \nabla (\phi (B^{i}_{s}) )^{2} ds, \quad \forall \, t \geq 0, \, \phi \in \mathscr{S}(\R^{d}).$$
By the central limit theorem, $\forall \, t \geq 0, \, \phi \in \mathscr{S}(\R^{d})$ we have $\Prob$-a.e.
$$ \langle M^{n}(\phi) \rangle_{t} \rightarrow \int_{0}^{t} \Exp \left[ (\nabla (\phi (B_{s}) )^{2} \right] ds, $$
where $(B_{t}: t\geq 0)$ is a $d$-dimensional Brownian motion starting at zero. 

Let $\mu_{t}(x)=(4 \pi t)^{-d/2} e^{-\norm{x}^{2}/4t}$ the density of a $d$-dimensional Gaussian random variable with mean zero and covariance matrix $2tI$. Define a mapping $A:[0,\infty) \times \mathscr{S}(\R^{d}) \rightarrow \R$ by 
$$A(t,\phi)\defeq \int_{0}^{t} \Exp \left[ (\nabla (\phi (B_{s}) )^{2} \right] ds = \int_{0}^{t}  ((\nabla \phi  )^{2} \ast \mu_{s})(0) \, ds, \quad \forall \, t \geq 0, \, \phi \in \mathscr{S}(\R^{d}). $$
Using Theorem  \ref{theoCentralLimitTheorem} we can show that the sequence $(M^{n}: n \in \N)$ converges weakly in 
 $C_{\infty}(\mathscr{S}'(\R^{d}))$. But before, we give a explicit definition for the limit. 

Let $W^{1}, \dots, W^{d}$ be $d$ independent $\mathscr{S}'(\R^{d})$-valued Wiener processes which are standard, i.e. $\Exp \left[ \abs{W^{i}_{t}(\phi)}^{2} \right] = t \norm{\phi}^{2}_{L^{2}(\R)}$ for every $t \geq 0$ and $\phi \in \mathscr{S}(\R^{d})$. Given $i=1, \dots, d$, let $F^{i}: [0,\infty) \rightarrow \mathcal{L}(\mathscr{S}'(\R^{d}),\mathscr{S}'(\R^{d}))$ given by $F^{i}(s)= ( \partial_{i} \sqrt{\mu_{s}}) + \sqrt{\mu_{s}} \partial_{i}$.   
Since for every $t>0$ and $\phi \in \mathscr{S}(\R^{d})$ we have
$$
\int_{0}^{t} \norm{F^{i}(s)'\phi}^{2}_{L^{2}(\R)} ds 
= \int_{0}^{t} \norm{-\sqrt{\mu_{s}} (\partial_{i}\phi)}^{2}_{L^{2}(\R)} ds = \int_{0}^{t} \int_{\R^{d}} \mu_{s}(x)(\partial_{i}\phi(x))^{2} dx ds< \infty,
$$
then $F^{i}$ is stochastically integrable with respect to $W^{i}$ (see \cite{FonsecaMora:SPDECMVM}). Hence, the $\mathscr{S}'(\R^{d})$-valued process $M^{0}=(M^{0}_{t}: t \geq 0)$ defined by
$$ M^{0}_{t}=\sum_{i=1}^{d} \int_{0}^{t} \, ( ( \partial_{i} \sqrt{\mu_{s}}) + \sqrt{\mu_{s}} \partial_{i}) dW_{i}, $$
is a Gaussian mean-zero continuous martingale satisfying 
$$ \Exp \left[ \abs{M^{0}_{t}(\phi)}^{2} \right] = \sum_{i=1}^{d}  \int_{0}^{t} \int_{\R^{d}} \mu_{s}(x)(\partial_{i}\phi(x))^{2} dx ds = \int_{0}^{t}  ((\nabla \phi  )^{2} \ast \mu_{s})(0) \, ds. $$
Then, by Theorem  \ref{theoCentralLimitTheorem} we have  $M^{n} \Rightarrow M^{0}$ in $C_{\infty}(\mathscr{S}'(\R^{d}))$. 

For every $n=0,1,2,\dots$, let $\eta^{n}$ be a $\mathcal{F}_{0}$-measurable random variable in $\mathscr{S}(\R^{d})'$ independent of $M^{n}$. Consider the lineal stochastic heat equation in $\mathscr{S}(\R^{d})'$:
\begin{equation}\label{eqHeatSEEForMn}
d Y^{n}_{t}=\Delta Y_{t}+ dM^{n}_{t}, \quad t \geq 0,
\end{equation}
with initial condition $Y^{n}_{0}=\eta^{n}$. Here  $\Delta$ is the Laplace operator on $\mathscr{S}(\R^{d})'$. 

Recall that $\Delta \in \mathcal{L}(\mathscr{S}(\R^{d}),\mathscr{S}(\R^{d}))$ is the infinitesimal generator of the \emph{heat semigroup}  $(S(t): t \geq 0)$ which is given by 
$S(t)\phi=\mu_{t} \ast \phi$ for each $\phi \in \mathscr{S}(\R^{d})$. The unique weak solution $Y^{n}=(Y^{n}_{t}: t \geq 0)$ to \eqref{eqHeatSEEForMn} satisfies that for any $t \geq 0$ and $\phi \in   \mathscr{S}(\R^{d})$:
\begin{equation*}\label{eqSoluHeatEqua}
\inner{Y^{n}_{t}}{\phi}= \inner{\eta^{n}}{\mu_{t} \ast \phi}  +\int_{0}^{t} \inner{M^{n}_{s}}{\Delta (\mu_{t-s} \ast \phi)} \, ds + \inner{M^{n}_{s}}{\phi}. 
\end{equation*}
If we assume that $\eta^{n} \Rightarrow \eta^{0}$, then by Theorem \ref{theoWeakConvSoluSEE} we conclude that $Y^{n} \Rightarrow Y^{0}$ in $C_{\infty}(\mathscr{S}'(\R^{d}))$. 
\end{example}



\smallskip

\noindent \textbf{Acknowledgements} The author would like to thank Adri\'an Barquero-S\'anchez for helpful suggestions  that contributed greatly to improve the presentation of this article.  This work was supported by The University of Costa Rica through the grant ``821-C2-132- Procesos cil\'{i}ndricos y ecuaciones diferenciales estoc\'{a}sticas''.








\begin{thebibliography}{HD}





\bibitem{BojdeckiGorostizaTalarczyk:2008} Bojdecki, T.; Gorostiza, L. G.; Talarczyk, A.: Occupation time limits of inhomogeneous Poisson systems of independent particles, \emph{Stochastic Process. Appl.}, {118}, no. 1, 28--52 (2008). 

\bibitem{Billingsley}  Billingsley, P.: \emph{Convergence of Probability Measures}, Wiley Series in Probability and Statistics, Wiley, second edition (1999). 

\bibitem{BudhirajaWu:2017} Budhiraja, A.; Wu, R.: Moderate deviation principles for weakly interacting particle systems, \emph{Probab. Theory Related Fields}, {168}, no. 3-4, 721--771 (2017). 
      
\bibitem{CohenElliott} Cohen, S. N.; Elliott, R. J.: \emph{Stochastic calculus and applications}, Second edition. Probability and its Applications, Springer (2015).
      
\bibitem{Chen:2019} Chen, Y-T.: Rescaled Whittaker driven stochastic differential equations converge to the additive stochastic heat equation, \emph{Electron. J. Probab.}, {24}, Paper No. 36, 33 pp (2019). 


\bibitem{DaleckyFomin} Dalecky, Y. L.;  Fomin, S. V.: \emph{Measure and Differential Equations in Infinite-Dimensional Space},  Mathematics and Its Applications 76, Springer Science+Business Media (1991).



\bibitem{FernandezGorostiza:1992} Fern\'{a}ndez, B.; Gorostiza, L. G.: Stability of a class of transformations of distribution-valued processes and stochastic evolution equations, \emph{J. Theoret. Probab.}, {5}, no.4, 661--678 (1992).

\bibitem{Fierro:1988} Fierro, R.:
Central limit theorems for local martingales taking values in the space of tempered distributions, \emph{Braz. J. Probab. Stat.}, no.2, 81--90 (1988).


\bibitem{FonsecaMora:Existence}  Fonseca-Mora, C. A.: Existence of Continuous and C\`{a}dl\`{a}g Versions for Cylindrical Processes in the Dual of a Nuclear Space, \emph{J Theor Probab}, {31}, no.2, 867--894 (2018). 

\bibitem{FonsecaMora:SPDECMVM} Fonseca-Mora, C. A.: Stochastic Integration and Stochastic PDEs Driven by Jumps on the Dual of a Nuclear Space, \emph{Stoch PDE: Anal Comp},   6, no.4,  618--689 (2018).

\bibitem{FonsecaMora:Skorokhod} Fonseca-Mora, C.A.: Tightness and Weak Convergence of Probabilities on the Skorokhod Space on the Dual of a Nuclear Space and Applications, \emph{Studia Math.}, {254}, no.2, 109--147 (2020).

\bibitem{FonsecaMora:StochInteg} Fonseca-Mora, C. A.: Stochastic Integration With Respect to Cylindrical Semimartingales, \emph{Electron. J. Probab.}, Volume 26, paper no. 147 (2021).

\bibitem{FonsecaMora:UCPConvergence} Fonseca-Mora, C. A.: Convergence Uniform on Compacts in Probability with Applications to Stochastic Analysis in Duals of Nuclear Spaces, \emph{Stoch. Anal. Appl.} Doi:10.1080/07362994.2024.2383678 



\bibitem{Fouque:1984}  Fouque, J-P.: La convergence en loi pour les processus \`{a} valeurs dans un espace nucl\'{e}aire, \emph{Ann. Inst. H. Poincar\'{e}}, {20}, no.3, 225--245 (1984). 


\bibitem{GubinelliTurra:2020} Gubinelli, M.; Turra, M.: Hyperviscous stochastic Navier-Stokes equations with white noise invariant measure, \emph{Stoch. Dyn.}, {20}, no. 6, 2040005, 39 pp (2020).


\bibitem{Jakubowski:1986}  Jakubowski, A.: On the Skorokhod topology, \emph{Annales de l'I.H.P. Probabilit\'{e}s et statistiques}, {22}, no.3 (1986), 263--285.


      
\bibitem{Jarchow} Jarchow, H.: \emph{Locally Convex Spaces}, Mathematische Leitf\"{a}den, Springer (1981). 



\bibitem{Kallenberg} Kallenberg, O.:  \emph{Foundations of modern probability}, Second edition, Probability and its Applications,  Springer-Verlag (2002). 


\bibitem{KallianpurXiong} Kallianpur, G.;  Xiong, J.: \emph{Stochastic Differential Equations in Infinite Dimensional Spaces}, Lecture Notes-Monograph Series, Institute of Mathematical Statistics (1995).




     
\bibitem{Kelley} Kelley, J. L.: \emph{General Topology},  D. Van Nostrand Co., Inc., Toronto-New York-London (1955).

\bibitem{Komura:1968} K\={o}mura, T.: Semigroups of operators in locally convex spaces, \emph{J. Funct. Anal.}, {2}, 258--296 (1968).



\bibitem{Lucon:2020} Lu\c{c}on, E.: Quenched asymptotics for interacting diffusions on inhomogeneous random graphs, \emph{Stochastic Process. Appl.}, {130}, no. 11, 6783--6842 (2020). 


\bibitem{Mitoma:1983UT} Mitoma, I.: Tightness of Probabilities On $C(\lbrack 0, 1 \rbrack; \mathscr{S}')$ and $D(\lbrack 0, 1 \rbrack; \mathscr{S}')$, \emph{Ann. Probab.}, {11}, no. 4, 989--999 (1983).      


\bibitem{Nakao:1986} Nakao, S.: On weak convergence of sequences of continuous local martingales, \emph{Ann. Inst. H. Poincaré Probab. Statist.}, {22}, no. 3, 371--380 (1986).

\bibitem{NariciBeckenstein}  Narici, L.; Beckenstein, E.: \emph{Topological Vector Spaces}, Pure and Applied Mathematics, CRC Press, second edition (2011). 


\bibitem{Pietsch} Pietsch, A.: \emph{Nuclear Locally Convex Spaces}, Ergebnisse der Mathematikund ihrer Grenzgebiete, Springer (1972).  

\bibitem{Rebolledo:1980} Rebolledo, R.: Central limit theorems for local martingales, \emph{Z. Wahrsch. Verw. Gebiete}, {51}, no. 3, 269--286 (1980).

\bibitem{Schaefer} Schaefer, H.: \emph{Topological Vector Spaces}, Graduate Texts in Mathematics, Springer, second edition (1999). 

     
\bibitem{Treves} Tr\`{e}ves, F.: \emph{Topological Vector Spaces, Distributions and Kernels}, Pure and Applied Mathematics, Academic Press (1967). 

     
\end{thebibliography}
\end{document}